\documentclass[reqno, 12pt, reqno]{amsart}

\usepackage{amsfonts, amsthm, amsmath, amssymb}

\usepackage{hyperref}

\RequirePackage{mathrsfs} \let\mathcal\mathscr

\numberwithin{equation}{section}

\newtheorem{theorem}{Theorem}[section]
\newtheorem{lemma}[theorem]{Lemma}

\newtheorem{corollary}[theorem]{Corollary}
\newtheorem{con}[theorem]{Conjecture}

\theoremstyle{definition}
\newtheorem*{ack}{Acknowledgements} 
\newtheorem*{notation}{Notation}

\renewcommand{\d}{\mathrm{d}}
\renewcommand{\phi}{\varphi}

\newcommand{\PP}{\mathbb{P}}
\renewcommand{\AA}{\mathbb{A}}

\newcommand{\ZZ}{\mathbb{Z}}

\newcommand{\NN}{\mathbb{N}}
\newcommand{\QQ}{\mathbb{Q}}
\newcommand{\RR}{\mathbb{R}}
\newcommand{\CC}{\mathbb{C}}

\newcommand{\cA}{\mathcal{A}}

\newcommand{\cD}{\mathcal{D}}
\newcommand{\cM}{\mathcal{M}}

\newcommand{\cN}{\mathcal{N}}

\renewcommand{\leq}{\leqslant}
\renewcommand{\le}{\leqslant}
\renewcommand{\geq}{\geqslant}
\renewcommand{\ge}{\geqslant}
\renewcommand{\bar}{\overline}

\newcommand{\m}{\mathbf{m}}

\newcommand{\w}{\mathbf{w}}
\newcommand{\x}{\mathbf{x}}
\newcommand{\y}{\mathbf{y}}

\renewcommand{\v}{\mathbf{v}}
\renewcommand{\u}{\mathbf{u}}

\renewcommand{\b}{\mathbf{b}}
\renewcommand{\a}{\mathbf{a}}

\renewcommand{\ss}{\mathfrak{S}}

\newcommand{\la}{\lambda}
\newcommand{\al}{\alpha}

\newcommand{\ve}{\varepsilon}

\newcommand{\bal}{\boldsymbol{\alpha}}

\newcommand{\bga}{\boldsymbol{\gamma}}

\DeclareMathOperator{\rank}{rank}

\DeclareMathOperator{\Pic}{Pic}

\DeclareMathOperator{\meas}{meas}

\DeclareMathOperator{\codim}{codim}
\DeclareMathOperator{\ann}{Ann}

\DeclareMathOperator{\Mod}{mod} 
\renewcommand{\bmod}[1]{\,(\Mod{#1})}

\newcommand{\Br}{{\rm Br}}

\newcommand{\minor}{\mathfrak{m}}
\newcommand{\major}{\mathfrak{M}}

\renewcommand{\t}{\mathbf{t}}

\renewcommand{\b}[1]{{\bf #1}}

\newcommand{\eeq}{\end{equation}}

\newcommand{\beql}[1]{\begin{equation}\label{#1}}
\newcommand{\un}[1]{\underline{#1}}
\newcommand{\QQb}{\bar{\QQ}}

\begin{document}

\title[Forms in many variables]{Forms in 
many variables\\ and differing degrees}
\author{T.D.\ Browning}
\author{D.R.\ Heath-Brown}

\address{School of Mathematics\\
University of Bristol\\ Bristol\\ BS8 1TW}
\email{t.d.browning@bristol.ac.uk}

\address{Mathematical Institute\\
Radcliffe Observatory Quarter\\ Woodstock Road\\ Oxford\\ OX2 6GG}
\email{rhb@maths.ox.ac.uk}

\date{\today}


\begin{abstract}
We generalise Birch's seminal work on forms in many variables to
handle a system of forms in which the degrees need not all be the
same. This allows us to 
prove the Hasse principle, weak approximation, and the Manin--Peyre 
conjecture for a smooth 
and geometrically integral 
variety $X\subseteq \PP^m$,  
provided only that its dimension  
is large enough in terms of its  degree. 
\end{abstract}

\maketitle
\setcounter{tocdepth}{1}
\tableofcontents

\section{Introduction and statement of results} \label{sec:intro}

This paper will be concerned primarily with integral solutions to
general systems of homogeneous equations
\beql{sys}
F_1(x_1,\dots,x_n)=\dots=F_R(x_1,\dots,x_n)=0,
\eeq
where each form $F_i$ has coefficients in $\ZZ$. Later in the paper we
will specialize our results to ``nonsingular systems'', and make 
deductions about the Hasse principle,  weak approximation and  the
distribution of rational points of bounded height,  for 
completely general smooth varieties.

Before describing the contents of the paper in detail, we would like
to state one particularly succinct result.

\begin{theorem}\label{thm:smooth}
Let $X\subseteq{\PP}^m$ be a smooth 
and geometrically integral 
variety defined over $\QQ$.  Then $X$ satisfies
the Hasse principle and weak approximation provided only that
\[\dim(X)\geq (\deg(X)-1)2^{\deg(X)}-1.\]
Moreover there is an asymptotic formula for
the counting function for $\QQ$-rational points of bounded  height on
$X$ which agrees with the  Manin--Peyre conjecture. 
\end{theorem}

The meaning of the final sentence will be made clear later in this introduction.

When $X$ is a hypersurface this theorem essentially reduces to a
well-known result 
of Birch \cite{birch}.  However we are able to handle varieties of
arbitrary codimension.  We would like to emphasize indeed 
that our 
hypotheses make 
no reference to the 
shape of the defining equations for $X$. In particular we
have not required $X$ to be a complete intersection.

It is rather striking that Theorem \ref{thm:smooth}  provides such
fine arithmetic information about the set $X(\QQ)$ of $\QQ$-rational
points on $X$ with such little geometric input.  
In the setting of hypersurfaces, for example, 
Harris, Mazur and Pandharipande 
\cite[\S~1.2.2]{HMP} have asked whether  
the above inequality already implies that $X$  
is {\em unirational},  meaning that there is a dominant rational map  
$\PP^{m-1}\rightarrow X$ defined  over $\bar\QQ$.   
In fact one of the main results in  \cite{HMP} 
shows that  there is an integer
$M(d)$ such that for $m\geq M(d)$ any smooth hypersurface 
$X\subseteq \PP^m$ 
of degree $d$ is indeed unirational. 
The value of $M(d)$ obtained is extremely large, and grows much faster 
than a $d$-fold iterated exponential of  $d$. 
It would be interesting to determine whether the methods of \cite{HMP}  
could be generalised to prove an analogous result for 
general smooth varieties. 

\medskip

Our principal tool will be the Hardy--Littlewood circle method, so
that we will be interested in the case in which the number of
variables is large. Our general problem has been considered by Schmidt
\cite{schmidt}, whose main result establishes the 
Hardy--Littlewood formula when the number of variables 
is sufficiently large in terms of certain 
``$h$-invariants''. Schmidt's work allowed him to deduce, for example,
that the system always has non-trivial solutions when the forms all 
have odd degrees, provided only that the number of variables is large
enough in terms of the 
degrees. The number  required is very large, but not as
large as in the original elementary proof of this result by Birch
\cite{odd}.  In general, while Schmidt's lower bound on the number of
variables required is explicit, the bound is quite awkward to compute, grows
rapidly, and depends on $h$-invariants which are very hard to
calculate. However, Schmidt also establishes a result (see
\cite[Corollary, page 262]{schmidt}) which is tolerably efficient for 
nonsingular systems, and which we will describe in a little more
detail later.  In the context of Theorem~\ref{thm:smooth} it would
produce a result when $n$ is very roughly of size $2^{3\deg(X)}$ or 
more. 

It is this second type of result that we wish to explore. Many of the
ideas go back to work of Birch
\cite{birch}. The method requires the system not to be too singular, but then
gives relatively good lower bounds for the number of variables
required. However Birch's original result needed the forms all to have the
same degree, and
there is a significant technical problem in extending
the method to the general case.  Schmidt showed how this might be
overcome, but his approach is somewhat wasteful, and does not recover
Birch's theorem in the case in which the forms all have the same
degree.  One of the
main purposes of this paper is to show how forms of unequal degrees can be
handled in an efficient manner, so as to give results in the spirit
of Birch \cite{birch} for arbitrary systems.

In order to describe Birch's result we introduce the {\em singular
locus} for the system of forms (\ref{sys}), which is the set
\[\{\x\in\AA^n:\rank(J(\x))<R\},\]
where $J(\x)$ is the Jacobian matrix of size
$R\times n$ formed from the
gradient vectors $\nabla F_1(\x),\dots, \nabla F_R(\x)$.  We note
that the system (\ref{sys}) defines an 
algebraic variety $V\subseteq\AA^n$. 
However, points of Birch's singular locus are not necessarily 
singular points of $V$, since they are not required to lie on $V$. 
If we write $B$ for the dimension (in $\AA^n$) of Birch's singular locus
then his theorem is that the usual Hardy--Littlewood formula holds as
soon as \beql{br}
n>B+R(R+1)(D-1)2^{D-1},
\eeq
where $D$ is the common degree of the forms $F_i$.

For our main result we will need a little more notation.
We will re-number the forms $F_i$ in (\ref{sys}), grouping together
those of equal degree.
Let $D\in\NN$ and let $r_d\in\NN\cup\{0\}$ for $1\le d\le D$, with
$r_D\ge 1$.   Suppose then that for every $d\le D$ we have forms
\beql{Fdn}
F_{1,d}(x_1,\dots,x_n),\dots,F_{r_d,d}(x_1,\dots,x_n)\in\ZZ[x_1,\dots,x_n]
\eeq
of degree $d$, so that  the total number of forms is
\[R=r_1+\dots+r_D.\]
In practice, if one had any forms of degree 1 it would be natural to
use them to eliminate appropriate variables, leaving a system of forms
of degrees at least 2 but involving fewer variables than originally.

It will be convenient to write
\[\Delta:=\{d\in\NN:\, r_d\ge 1\}\subseteq\{1,2,\dots,D\}.\]
For each degree $d\in\Delta$ we define the matrix
\[J_d(\x):=\left(\begin{array}{c}\nabla F_{1,d}(\x)\\ \vdots \\
\nabla F_{r_d,d}(\x)\end{array}\right)\]
and we set
\[S_d:=\{\x\in\AA^{n}:\,\rank(J_d(\x))<r_d\}.\]
This defines an affine algebraic variety and we henceforth set
\begin{equation}\label{eq:sig-d}
B_d:=\dim(S_d).
\end{equation}
When $r_d=0$ we shall take $B_d=0$. It will also be convenient to set
$B_0=0$.  Our method breaks down if there is
any degree $d$ for which $B_d=n$, and so we impose the condition that
$B_d<n$ for every $d\in\Delta$. For example, this rules out the case
in which the forms (\ref{Fdn}) are linearly dependent. 

At this point we should observe that forthcoming independent work of Dietmann
\cite{D} and Schindler \cite{S}
allows one to replace $B_d$ by an alternative invariant, which we
denote temporarily by $B_d'$.  One can show in complete generality that
$B_d'\le B_d$, but that $B_d'$ can be strictly less than $B_d$ in
appropriate cases.  However we
will work with Birch's invariant $B_d$ throughout this paper.

We wish to count integral vectors in a fixed congruence 
class, and which lie in the dilation of a fixed box.  
We therefore choose an $n$-dimensional box  $\mathcal{B}\subseteq [-1,1]^n$,
with sides aligned to the coordinate axes.
We also give ourselves 
a modulus $M\in\NN$ and a vector $\b{m}_0\in\ZZ^n$ 
with coordinates in $[0,M-1]$. 
The  box $\mathcal{B}$, the modulus $M$ and the vector $\b{m}_0$ will
be considered fixed.  
For any (large) 
positive real 
$P$ we then write
\[N(P):=
\#\{\x=\b{m}_0+M\y:\,\y\in\ZZ^n,\,\x\in P\mathcal{B},  
\,F_{i,d}(\x)=0\, \forall i,d\}.\]
The vectors $\x$ which occur here all satisfy
$\x\equiv\b{m}_0\bmod{M}$. 
Typically we will want to choose the box $\mathcal{B}$ so that the
vectors $\x$ lie close (in a projective sense) to a given real point.  
Suppose  we have chosen a non-zero vector $\x_0\in (-1,1)^n$  and a
small positive constant $\eta.$ 
Taking $|\x|$ to denote the sup-norm of the vector $\x$ and setting 
$$
\mathcal{B}=\{\mathbf{u}\in \RR^n: |\mathbf{u}-\x_0|<\eta \},$$  
we see that 
$P^{-1}\x$ will be 
close $\x_0$ whenever  
$\x$ is counted by $N(P)$. 

Unfortunately the condition for $n$ occurring in our first result is rather
complicated.  We put \begin{equation}\label{eq:Dj}
\mathcal{D}_j:=r_1+2r_2+\dots+jr_j,
\end{equation}
for $1\leq j\leq D$, and we set $\mathcal{D}_0:=0$ and
$\mathcal{D}:=\mathcal{D}_D$. Finally we write 
\begin{equation}\label{eq:def-sj}
s_d:=\sum_{k=d}^D \frac{2^{k-1}(k-1)r_k}{n-B_k}.
\end{equation}

With these conventions we now have the following.
\begin{theorem}\label{thm:main}
Suppose we have
\[\mathcal{D}_d\left(\frac{2^{d-1}}{n-B_d} +s_{d+1}\right)
+s_{d+1}+\sum_{j=d+1}^D s_jr_j<1\] 
for $d=0$ and for every $d\in\Delta$.
Then there is a positive $\delta$ such that
\[N(P)=\sigma_{\infty}\left(\prod_p\sigma_p\right)P^{n-\cD}
+O(P^{n-\cD-\delta}),\]
where $\sigma_{\infty}$ and $\sigma_p$ are the usual local densities,
given by \eqref{eq:local-inf}  and \eqref{eq:local-p}, respectively.
\end{theorem}
Here, and for the rest of the paper, the implied constant is allowed
to depend on the forms $F_{i,d}$ (and hence on $n$, $R$ and $\cD$) and also
on the box $\mathcal{B}$, the 
modulus $M$ and the vector $\b{m}_0$.

We observe at this point that the entire analysis may be applied to
systems of polynomials $f_{i,d}$, rather than systems of forms. For
each such polynomial one defines the form $F_{i,d}$ to be the
homogeneous part of $f_{i,d}$ of degree $d$.  One then uses the
various 
$F_{i,d}$ to define the numbers $B_d$ as before. The
entire argument now goes through with only minor modifications.

Although our condition on $n$ is somewhat complicated the reader may
readily verify that if $r_1=\dots=r_{D-1}=0$ and $r_D=R$, then it
is equivalent to  Birch's constraint 
in (\ref{br}).  In order to understand better our condition we give
the following corollary of Theorem \ref{thm:main}, which is simpler
but potentially weaker.

\begin{corollary}\label{simple}
Write
\[B:=\max\{B_d:\,d\in\Delta\}\]
and set
\[t_d:=\sum_{k=d}^D 2^{k-1}(k-1)r_k,\;\;\;(1\le d\le D+1),\]
\[n_0(d):=\mathcal{D}_d\left(2^{d-1} +t_{d+1}\right) +t_{d+1}+\sum_{j=d+1}^D t_jr_j\]
and
\[n_0:=\max\{n_0(d):\,d\in\Delta\cup\{0\}\}.\] 
Then the conclusion  of Theorem \ref{thm:main} holds whenever
$n>B+n_0$.
\end{corollary}

For comparison, the result of Schmidt \cite[Corollary, page
262]{schmidt} mentioned before would establish the same conclusion as
Theorem \ref{thm:main} as soon as
\[n>
\max_{d\le D}\left(B_d+(d-1)(1+2^{1-d})^{-1}2^{3d-5}r_dD\mathcal{D}\right).\]

As examples of Corollary \ref{simple}
we proceed to consider some  test cases.

\begin{corollary}\label{cor:2D}
For a system consisting of $r\ge 1$ quadratic forms and a single form of degree
$D\ge 3$ we have $n_0=(2+r)(D-1)2^{D-1}+2r(r+1)$ when $r>
(D-1)2^{D-2}$, and $n_0=(2+2r)(D-1)2^{D-1}+4r$ otherwise.
\end{corollary}

Thus if $D$ is fixed and $r$ tends to infinity our bound is asymptotic
to the value $2r(r+1)$ we would have for a system consisting solely
of quadratic forms. On the other hand, when $r$ is fixed and $D$ grows
we do not get a bound asymptotic to the value $(D-1)2^{D}$ we would
have for a single form of degree $D$.

The proof of Corollary \ref{cor:2D} is a straightforward calculation.
We find that
\begin{align*}
n_0(D)&=(D+2r)2^{D-1},\\
n_0(2)&=(2+2r)(D-1)2^{D-1}+4r
\end{align*}
and 
$$
n_0(0)=(2+r)(D-1)2^{D-1}+2r(1+r).
$$
Hence $n_0(D)\le n_0(0)$ for every value of $r$ and $n_0(0)\geq n_0(2)$ if and
only if $r> (D-1)2^{D-2}$.

\begin{corollary}\label{cor:DD}
For a system consisting of one form of degree $D$ and one of degree $E$,
where $D>E\ge 2$, we have \[n_0=(2+E)(D-1)2^{D-1}+E2^{E-1}.\]
\end{corollary}

In particular, if $E\ge 4$ then we have a larger value for $n_0$ than
for a system consisting of two forms of degree $D$.  This is slightly
disappointing, since one would expect that is is ``easier'' to handle a
pair of forms of degrees 4 and 5, say, than two forms of degree 5.

Again the proof of Corollary 
\ref{cor:DD}
 is a straightforward calculation.
This time we find that
\begin{align*}
n_0(D)&=(D+E)2^{D-1},\\
n_0(E)&=(2+E)(D-1)2^{D-1}+E2^{E-1}
\end{align*}
and
\[n_0(0)=3(D-1)2^{D-1}+2(E-1)2^{E-1},\]
and one readily checks that $n_0(E)$ is at least as large as $n_0(D)$
or $n_0(0)$. 

In general we can give the following crude upper bound for $n_0$.
\begin{theorem}\label{thm:crude}
We have
\[n_0+R-1\le \cD^2 2^{D-1}\le R^2D^2 2^{D-1}\]
and
\[n_0+R-1\le (\cD-1)2^{\cD}.\]
\end{theorem}

Many variants of this are possible.  We have chosen to give an
estimate with a term $R-1$ on the left because there is a significant 
case in which one has $\max B_d\le R-1$, as we shall see below. 

The first bound shows in particular that for any system of $R$ forms of
degrees at most $D$ one has $n_0\ll_D R^2$.  
A result of this type, with a somewhat worse
dependence on $D$, was first proved by Schmidt \cite[Corollary, page 262]{schmidt}.

In order to give more information about the dimensions $B_d$ of
Birch's singular loci we shall now investigate what happens if we
impose a nonsingularity condition. This will also enable us to
describe conditions under which 
the constant $\sigma_{\infty}\prod_p\sigma_p$ is positive
in Theorem \ref{thm:main}. 
 We
shall say that the collection of forms $F_{i,d}$ is a {\em nonsingular system} if
$\rank(J(\x))=R$ for every non-zero $\x\in\QQb^n$ satisfying the
equations
\beql{sys1}
F_{i,d}(\x)=0, \quad  (1\le i\le r_d,\,1\le d\le D),
\eeq
where $J(\x)$ is the $R\times n$ Jacobian matrix defined above.

In order to get good bounds on $B_d$ we replace our system of forms by an
``equivalent optimal system''.  We shall say that two systems
$F_{i,d}$ and $G_{i,d}$ of integral forms (with
$\deg(F_{i,d})=\deg(G_{i,d})=d$) are {\em equivalent} if for every pair
$i,d$ the form $F_{i,d}-G_{i,d}$ is a linear combination 
\[\sum_{j<i}H_{j,d}(\x)F_{j,d}(\x)+     
\sum_{e<d}\sum_{j\le r_e}H_{j,e}(\x)F_{j,e}(\x) \]  
where $H_{j,e}$ is an integral form of degree $d-e$. 
One sees at 
once that this does indeed produce an equivalence relation, and that
the forms $G_{i,d}$ have the same set of zeros as the original system $F_{i,d}$.

We shall prove in Section \ref{s:geom}
that if one has a nonsingular system of forms $\{F_{i,d}\}$, 
then there is an equivalent system $\{G_{i,d}\}$ with the 
property that for any value of $i$ and $d$ the sub-system 
\[\{G_{j,d}:\, j\ge i\} \cup \{G_{j,e}:\, j \le r_e,\, d< e\le D\}\] 
is itself a nonsingular system.  We call such a
system an {\em optimal system}.  For example, if our original
nonsingular system consists of a cubic form $C$ and a quadratic form
$Q$, then there will be a linear form $L$ such that $C+LQ$ is a
nonsingular form.  The pair $\{C+LQ,Q\}$ is then an optimal system.

For an optimal system we shall show in Lemma \ref{lem:improve} that  
\beql{eq:rhul}
B_d\le r_d+\dots+r_D-1, \quad (1\le d\le D).
\eeq
It follows that $\max B_d\le R-1$ for an optimal nonsingular 
system. Since equivalent systems have the same counting function 
$N(P)$ we therefore deduce the following result. 

\begin{theorem}\label{thm:main2}
Suppose we have a nonsingular system of forms such that
$n>(\cD-1)2^{\cD}$.  Then there is a positive $\delta$ such that 
\[N(P)=\sigma_{\infty}\left(\prod_p\sigma_p\right)P^{n-\cD}
+O(P^{n-\cD-\delta}),\]
where $\sigma_{\infty}$ and $\sigma_p$ are the usual local densities,
given by \eqref{eq:local-inf}  and \eqref{eq:local-p}, respectively.
 Moreover
$\sigma_{\infty}$ is positive provided  that the
system of equations \eqref{sys1} has a real solution
in $\mathcal{B}$. 
Similarly $\prod_p\sigma_p$ is positive provided that
for each prime $p$ there is a solution $\x_p\in\ZZ_p^n$ satisfying
$\x_p\equiv\b{m}_0\bmod{M}$.
\end{theorem}

We show in Section \ref{sec:major} that the singular series and 
singular integral are absolutely convergent under the conditions 
of Theorem \ref{thm:main}.  Thus standard arguments, such as those
used by Davenport \cite[Chapters 16 \& 17]{DavAM}, show that they are
positive whenever suitable nonsingular local solutions exists. 
The details are left to the reader.

The bound (\ref{eq:rhul}) also enables us to establish the  
following variant of Corollary \ref{cor:DD}.

\begin{corollary}\label{cor:DD+}
For a nonsingular
system consisting of one form of degree $D$ and one of degree $E$,
where $D>E\ge 2$, the conclusion of Theorem \ref{thm:main2} holds whenever
\[n>(2+E)(D-1)2^{D-1}+E2^{E-1}.\]
\end{corollary}

In the case of one quadratic and one cubic we find that $n\ge 37$
suffices. This reproduces one of the results from the work of
Browning, Dietmann and Heath-Brown \cite{bdhb}.  However in this
special case one can do better.  Indeed it is shown in
\cite[Theorem 1.3]{bdhb} 
that one can handle smooth intersections of 
one quadratic and one cubic as soon as $n\ge 29$.

To prove the corollary one has merely to interpret the condition of Theorem
\ref{thm:main} subject to the information in 
\eqref{eq:rhul}. One
therefore needs
\[\frac{(D+E)2^{D-1}}{n}<1,\] 
\[\frac{(2+E)(D-1)2^{D-1}}{n}+\frac{E2^{E-1}}{n-1}<1\] 
and
\[\frac{3(D-1)2^{D-1}}{n}+\frac{2(E-1)2^{E-1}}{n-1}<1,\]
corresponding to $d=D,E$ and $0$, respectively. It is easy to see that 
 $(2+E)(D-1)\ge
D+E$ whenever $D>E\ge 2$, so that the second condition implies the
first. In general, if $\alpha$ and $\beta$ are positive 
integers one has
\[\frac{\alpha}{n}+\frac{\beta}{n}<\frac{\alpha}{n}+\frac{\beta}{n-1}
<\frac{\alpha}{n-1}+\frac{\beta}{n-1},\]
so that the inequality
\[\frac{\alpha}{n}+\frac{\beta}{n-1}<1\]
will hold for $n=\alpha+\beta+1$, but not for $n=\alpha+\beta$.
Since
\[(2+E)(D-1)2^{D-1}+E2^{E-1}\ge 3(D-1)2^{D-1}+2(E-1)2^{E-1},\]
we therefore see that the condition in Theorem \ref{thm:main} holds if
and only if 
\[n\geq (2+E)(D-1)2^{D-1}+E2^{E-1}+1,\]
and the result follows.

\medskip

Up to this point we have described our results in terms of zeros of
systems of forms.  We now turn to the related question of rational
points on projective varieties. 
Recall that a  family of projective 
algebraic varieties $X$, each defined over $\QQ$, is 
said to satisfy the {\em Hasse principle} if $X$ has a point over
$\QQ$ 
whenever it has a point over each completion of $\QQ$. If in addition
the set $X(\QQ)$ of $\QQ$-points of $X$ is dense in the ad\`{e}lic
points then we say 
that {\em weak approximation} holds.
When $X$ is Fano (i.e. it is 
a nonsingular projective variety with ample anticanonical bundle $\omega_X^{-1}$)
and  $X(\QQ)$ is dense in $X$ under the Zariski topology,
it is natural to study the counting function
$$
N(U,H,P)=\#\{x\in U(\QQ): H(x)\leq P\},
$$
as  $P\rightarrow \infty$. Here  $U\subseteq X$ is any Zariski open subset 
and  $H$ is  any anticanonical height function on $X$.
The {\em Manin--Peyre conjecture} (see \cite{manin} and \cite{peyre})
predicts the existence of an open subset $U\subseteq X$
such that for any anticanonical height function $H$ on $X$ there is
a (precisely described) constant $c_{U,H}>0$
such that
\begin{equation} \label{con}
	N(U,H,P) \sim c_{U,H} P(\log P)^{\rank \Pic(X)-1}, \quad (P \to \infty).
\end{equation}
We will be interested in this when 
$U=X$
and 
$\Pic(X)\cong \ZZ$.

Any smooth complete intersection in
$\PP^{n-1}$ is the zero-set of a nonsingular system of
forms. Conversely the equations (\ref{sys1}) define a 
variety, $X$ say, in $\PP^{n-1}$. We shall prove in Lemma \ref{lem:NS}
that if one has a 
nonsingular system, then the corresponding variety $X$  is
geometrically integral,  
and indeed the ideal in $\QQb[\x]$ which annihilates $X(\QQb)$
is generated by the forms $F_{i,d}$. In particular $X$ is 
smooth.  Moreover we will show that $X$ has 
codimension $R$ in $\PP^{n-1}$, and that its degree is 
\[\deg(X)=\prod_{d\le D}d^{r_d}.\]

Recall that $X\subseteq\PP^{n-1}$ is said to be non-degenerate if it
is not contained in any proper linear subspace of $\PP^{n-1}$. In
this case we must have $r_1=0$, whence one easily finds that
$\deg(X)\ge\mathcal{D}$. In view of Theorem \ref{thm:main2} we can
therefore handle any smooth 
 non-degenerate complete intersection $X\subseteq\PP^{n-1}$ for which
\beql{n>}
n>(\deg(X)-1)2^{\deg(X)}.
\eeq
We claim that the Hasse principle and weak
approximation hold for such varieties, together with the  Manin--Peyre
conjecture with $U=X$.  
Taking the lower bound $\deg(X)\geq \mathcal{D}\geq 2R$,   the inequality \eqref{n>}  implies that 
 $\dim(X)=n-1-R\geq 3$.
In particular 
 the natural map $\Br(\QQ)\rightarrow \Br(X)$ is an isomorphism  (see 
Proposition~A.1 in Colliot-Th\'el\`ene's  appendix to \cite{PV}), where $\Br(X)=H_{\tiny{\mbox{\'et}}}^2(X,\mathbb{G}_m)$ is the  Brauer group of $X$. Hence this is compatible with the conjecture of Colliot-Th\'el\`ene that the {\em Brauer--Manin obstruction} controls the Hasse principle and weak approximation for the varieties under consideration here (see \cite{bud} for the most general statement of this conjecture).

To see the claim, we observe that the Hasse principle and weak approximation  
follow on choosing $\mathcal{B}$ so that the vectors counted by $N(P)$
lie close to a given real point on $X$ and letting $P$ run through
large positive integers.  For the Manin--Peyre conjecture
with  $U=X$,  we may assume that $X(\QQ)\neq \emptyset$. It follows from 
\cite[\S II, Exercise~8.4]{hart} that
$\omega_X^{-1}=\mathcal{O}(n-\mathcal{D})$ and the inequality 
 \eqref{n>} ensures that   $X$ is Fano. Moreover
  $\Pic X \cong  \ZZ$ by Noether's theorem (see  \cite[Corollary~3.3, page~180]{har70}).
We  work with the height function
$$H(x) = \|\x\|^{n-\mathcal{D}},$$ where   $\|\cdot\|$ is an 
arbitrary norm on $\RR^{n}$, 
on choosing  a representative $x=[\x]$ such that $\x\in \ZZ^n$ is  primitive.
Put  $C=\sup_{\x\in [-1,1]^n}\|\x\|$ and
\[\mathcal{R}=\{\x\in \RR^n: \|\x\|\leq C\}\subseteq [-1,1]^n.\] 
In order to establish \eqref{con}, it turns out that it is 
enough to estimate $N(P)$, with $M=1$ and the box $\mathcal{B}$ 
replaced by the region $\mathcal{R}$. 
In effect one counts integral points of bounded height on the 
universal torsor over $X$. (Note that 
the affine cone over $X$ in $\mathbb{A}^{n}\setminus \{\mathbf{0}\}$
is the unique universal torsor over $X$ up to isomorphism since
$\dim(X)\geq 3$.) 
Although $\mathcal{R}$ is not necessarily a box it can be approximated
arbitrarily closely, both from above and below, by a disjoint union of 
admissible boxes. The 
desired asymptotic formula for  
$N(P)$ now follows from Theorem \ref{thm:main2}.

It has been observed that there are no examples in the literature in
which the Hardy--Littlewood circle method has been used for
varieties which are not complete intersections. Indeed there has been
speculation that the circle method is incapable of handling such
varieties.  Of course, it is not easy to formalize such a claim.

However, one reason that the circle method has been applied only to
complete intersections is that it requires the dimension to be large
relative to the degree, as one sees in Birch's result (\ref{br}) for example.
In contrast, varieties which are not complete intersections tend to
have dimension which is at most of size comparable with the degree.
Indeed Hartshorne \cite{hart-conj} has conjectured that a smooth 
variety $X\subseteq\PP^m$ is a complete intersection as
soon as $\dim(X)>2m/3$.  According to Harris \cite[Corollary~18.12]{harris} 
any variety $X\subseteq\PP^m$ lies in a linear 
subspace of dimension at most $\dim(X)+\deg(X)-1$, and if $X$ is defined
over $\QQ$ we can take the subspace also to be defined over
$\QQ$. Thus in our context we may assume that $m\le\dim(X)+\deg(X)-1$, so that
Hartshorne's conjecture implies that $X$ is a complete intersection as
soon as \[\dim(X)>\frac{2}{3}\left(\dim(X)+\deg(X)-1\right),\]
or equivalently, whenever 
\begin{equation}\label{eq:hartshorne}
\dim(X)\ge 2\deg(X)-1.
\end{equation}
If this were true it would
certainly explain why we have no examples where the circle method has
handled a variety which is not a complete intersection.

Hartshorne's conjecture is still largely wide open. However, it has
been shown by Bertram, Ein and Lazarsfeld 
\cite[Corollary~3]{ein} 
that if $X\subseteq \PP^m$ is smooth
then it is a complete intersection as soon as 
\[\deg(X)\le \frac{m}{2(m-\dim(X))}.\] 
We may assume as above that $m\le \dim(X)+\deg(X)-1$.  Inserting this
information into the above inequality and rearranging
we conclude that $X$ is a complete intersection provided only that
\[\dim(X)>\deg(X)(2\deg(X)-3).\]
This enables us to deduce Theorem \ref{thm:smooth} from Theorem
\ref{thm:main2}. 
We observe
 firstly that the  result is  trivial if $X$ is linear. 
Otherwise, if $X$ is as in Theorem \ref{thm:smooth}, then it 
lies in a minimal linear space $L$ say, defined over $\QQ$. 
If we write 
 $n-1=\dim(L)>\dim (X)$,
then $X$ is a smooth, non-degenerate, geometrically integral  
subvariety of  $L\cong \PP^{n-1}$. 
 Moreover, we have 
$n-1>(\deg(X)-1)2^{\deg(X)}-1$. 
Under the hypothesis of  Theorem~\ref{thm:smooth}, $X$ will 
be a complete intersection, by the result of
Bertram, Ein and Lazarsfeld, since we have
\[(\deg(X)-1)2^{\deg(X)}-1>\deg(X)(2\deg(X)-3)\]
for $\deg(X)\ge 2$.  
Moreover, we shall prove in Lemma \ref{lem:new} that the annihilating
ideal of $X$ is generated  
by integral forms. 
The result  then follows since we have already 
observed that (\ref{n>}) suffices for smooth non-degenerate complete
intersections defined over $\QQ$.

\medskip

We conclude this introduction by discussing 
the extent to which one might relax the conditions of
Theorem~\ref{thm:smooth}.

\begin{con}\label{conjecture}
Let $X\subseteq \PP^m$ be a smooth and geometrically integral variety
defined over $\QQ$. Then $X$ satisfies the Hasse principle and weak
approximation  provided only that $\dim (X)\geq 2 \deg(X)-1$.
Moreover, if $X(\QQ)\neq \emptyset$,  the  Manin--Peyre conjecture
holds with $U=X$. 
\end{con}

The conclusion of the conjecture is trivial if $\deg(X)=1$ 
and well-known for $\deg(X)=2$ 
and so we may assume that $\deg(X)\geq 3$. 
In particular $\dim (X)\geq 5$. 
In this case the  first part of the conjecture is based on combining
the conjectures of Hartshorne and Colliot-Th\'el\`ene that we
mentioned above. According to the former, the inequality
\eqref{eq:hartshorne} is enough to ensure that any $X$ in the
statement of Conjecture~\ref{conjecture} is a complete intersection in
$L$,  for some linear subspace $L\cong \PP^{n-1}\subset \PP^m$.  
Assuming that $X$ is defined by a system of $R$ equations \eqref{sys},
we deduce   
that  $X$ is Fano since 
\begin{equation}
n>\dim (X)+1 \geq 2\deg(X)\geq 2\mathcal{D}.
\label{eq:gbk}
\end{equation}
Hence Colliot-Th\'el\`ene's conjecture implies that $X$ satisfies the
Hasse principle and weak approximation  
(see \cite[Conjecture 3.2 and Proposition A.1]{PV}).
Finally, the inequality \eqref{eq:gbk} is precisely what arises from 
the  ``square-root barrier'' in the circle method, with the general
expectation then being that  
the usual  Hardy--Littlewood formula ought to hold, provided that  $X$ is smooth
 and geometrically integral. 
As above this would lead to a resolution of the Manin--Peyre
conjecture  with $U=X$.

We close by discussing two examples to illustrate Theorem
\ref{thm:smooth} and Conjecture \ref{conjecture}.     
Suppose that $m=2d-1$ and consider the Fermat hypersurface 
\[X: \quad x_0^d+\dots+x_{d-1}^d=x_d^d+\dots+x_{2d-1}^d\]
in $\mathbb{P}^m$. 
Note that  $X$ contains the $(d-1)$-plane given by the equations 
$$x_i=x_{i+d}, \quad \mbox{for $i=0,\dots,d-1$}.
$$  
It was shown by Hooley \cite{hooley} that this variety has more points
than the circle method leads one to expect.  Indeed it follows from work of 
Browning and  Loughran \cite[Example~3.2]{BL} that there is at least
one choice of anticanonical height
function for which the Manin--Peyre conjecture fails when $U=X$.
This example shows that we cannot 
have a result like
Theorem~\ref{thm:smooth} in which the condition is relaxed to 
$\dim(X) \ge 2\deg(X)-2$. Thus the  lower bound in Conjecture~\ref{conjecture} is optimal, from the point of view of the Manin--Peyre conjecture.

Turning to the question of the Hasse principle, for any $k\in\NN$ we
consider the variety $X\subseteq\PP^{3k+2}$ defined as follows. Let
$C\subseteq\PP^2$ be the curve given by $3x_1^3+4x_2^3+5x_3^3=0$, and
let $\phi:\PP^2\times\PP^k\rightarrow\PP^{3k+2}$ be the Segre
embedding.  Then we take $X$ to be $\phi(C\times\PP^k)$.  It is easy
to see that $X$ fails the Hasse principle since $C$ fails the Hasse principle. 
Moreover $\deg(X)=3(k+1)$, as in Harris \cite[pages 239 \&  
240]{harris}, and $\dim(X)=k+1$.  
Finally $X$ is smooth, as  in 
Hartshorne \cite[Proposition~III.10.1(d)]{hart}. Thus Theorem~\ref{thm:smooth}
would be false if the lower bound on $\dim(X)$ were replaced by
$\tfrac13\deg(X)$. It would be interesting to have examples of the
failure of the Hasse principle in which $\dim(X)$ grows faster than
$\tfrac13\deg(X)$.

\begin{notation}
For any $\al\in \RR$, we will follow common convention and
write $e(\al):=e^{2\pi i\al}$ and $e_q(\al):=e^{2\pi i\al/q}$.  
We will allow all of our implied constants to depend on
$\ve$, in addition  to the forms $F_{i,d}$ and the objects
 $\mathcal{B}$,  $M$ and  $\mathbf{m}_0$ occurring in the definition of
 $N(P)$. 
 We shall write $|\x|$ for the sup-norm of a vector
 $\x\in\CC^n$ and we use $\|\theta\|$ for the distance from a real
 number $\theta$ to the nearest integer.
Finally, we shall often write $\a=(a_{i,d})$ to denote the vector
whose $R$ entries are indexed by $i,d$ satisfying 
$1\leq i\leq r_d$ and $1\leq d\leq D$.
\end{notation}

\begin{ack} 
While working on this paper the first author
was  supported by ERC grant \texttt{306457}. The authors are very
grateful to Julia Brandes,  Daniel Loughran and the anonymous referee for their comments on an
earlier draft of this paper, and to Christopher Frei for pointing out
an error in our original treatment of Lemma \ref{lem:majorA}.
\end{ack}

\section{Overview of the paper}\label{s:2}

The aim of the present section is to present the main ideas in the
proof of Theorem \ref{thm:main}, which is the principal result in this
paper. 
The starting point in the circle method is the  identity
$$
N(P)=\int_{(0,1]^R} S(\bal) \d\bal,
$$
where $\bal=(\alpha_{i,d})$ for $1\leq i\leq r_d$ and $1\leq d\leq D$, and 
$$
S(\bal):=\sum_{\substack{\x\in \ZZ^n\\ \m_0+M\x\in P\mathcal{B}}}
e\left(\sum_{d=1}^D\sum_{i=1}^{r_d}\al_{i,d} F_{i,d}(\m_0+M\x)\right).
$$
The idea is then to divide the region $(0,1]^R$ into a
set of major arcs $\major$ and minor arcs $\minor$. In the usual way
we wish  to prove an asymptotic
formula \begin{equation}\label{eq:major} 
\int_{\major} S(\bal) \d\bal=
\sigma_\infty \left(\prod_p \sigma_p\right) P^{n-\mathcal{D}}+O(P^{n-\mathcal{D}-\delta}),
\end{equation}
for some $\delta>0$, together with a satisfactory bound on the minor arcs
\begin{equation}\label{eq:minor}
\int_{\minor} S(\bal) \d\bal
=O(P^{n-\mathcal{D}-\delta}).
\end{equation}
In the above formula the real density associated to the counting problem 
described by $N(P)$ is defined to be  
\begin{equation}\label{eq:local-inf}
\sigma_\infty:=
\frac{1}{M^n} 
\int_{\RR^{R}} J(\bga) \d\bga,
\end{equation}
where 
\begin{equation}\label{eq:defn-J}
J(\bga):=\int_{\mathcal{B}}
e\left(\sum_{d=1}^D\sum_{i=1}^{r_d}\gamma_{i,d} F_{i,d}(\x)\right)
\d\x. 
\end{equation}
The corresponding $p$-adic density is 
\begin{equation}\label{eq:local-p} 
\sigma_p:=~\lim_{k\rightarrow \infty} p^{-(n-R)k}\mathcal{N}(p^k)
\end{equation}
where
\[\mathcal{N}(q):=\#\left\{\x\in (\ZZ/q\ZZ)^n : 
F_{i,d}(\m_0+M\x)\equiv 0 \bmod{q} ~\forall i,d  \right\}.\]

Let $\varpi\in(0,1/3)$ be a parameter to be decided upon later (see
equation \eqref{eq:choosedelta}). We will take as major arcs
$$
\major:=\bigcup_{q\leq P^{\varpi}\;}
\bigcup_{\substack{\a\bmod{q}\\ 
    \gcd(q,\a)=1}}\major_{q,\a}, 
$$
where $\a=(a_{i,d})$ and 
\beql{majdef}
\major_{q,\a}:=\left\{\bal\bmod{1}: \begin{array}{l}
\left|\alpha_{i,d}-\frac{a_{i,d}}{q}\right|\leq
P^{-d+\varpi} \mbox{ for}\\ 
\mbox{$1\leq i\leq r_d$ and $d\in\Delta$}
\end{array}{}
\right\}. \eeq
We have  $\major_{q,\a}\cap \major_{q',\a'}=\emptyset$
whenever  $\a/q\neq \a'/q'$, provided that $P$ is taken to be
sufficiently large.

The  minor arcs are defined to be $\minor=(0,1]^R\setminus \major$.
Our estimation of $S(\bal)$ for $\bal\in \minor$
is based on a version of Weyl differencing,
which is inspired by the work of Birch \cite{birch}, but which
is specially adapted to systems of forms of differing degree.

For each $d\in\Delta$ let $F_{i,d}(\x_1,\dots,\x_d)$ be the $d$-multilinear
polar form attached to $F_{i,d}(\x)$. After multiplying $F_{i,d}$ by
$d!$ we may assume that $F_{i,d}(\x_1,\dots,\x_d)$ has integer coefficients.
We take
$\un{F}_{i,d}(\x_1,\dots,\x_{d-1})$ to be the row vector for which
\beql{dot}
F_{i,d}(\x_1,\dots,\x_d)=\un{F}_{i,d}(\x_1,\dots,\x_{d-1}).\x_d,
\eeq
and we set \[\widehat{J}_d(\x_1,\dots,\x_{d-1})=
\left(\begin{array}{c}
\un{F}_{1,d}(\x_1,\dots,\x_{d-1})\\ \vdots \\
\un{F}_{r_d,d}(\x_1,\dots,\x_{d-1})\end{array}\right)\]
and
\begin{equation}\label{eq:hat-S}\widehat{S}_d=
\{(\x_1,\dots,\x_{d-1})\in(\AA^{n})^{d-1}:\,
\rank(\widehat{J}_d(\x_1,\dots,\x_{d-1}))<r_d\}.
\end{equation}
Thus $\widehat{S}_d$ is an affine algebraic variety.

Using $D-1$ successive applications of Weyl
differencing,  as in Birch's work, we can relate the size of the exponential sum
$S(\bal)$ to the locus of integral points on the affine variety
$\widehat S_D$. In this way we shall be able to get good control over $S(\bal)$
unless $\alpha_{1,D},\dots,\alpha_{r_D,D}$ all happen to be close to a
rational number with small denominator.
If this occurs then we shall modify the final Weyl squaring,
in a way suggested by the ``$q$-analogue'' of van der Corput's method,
so as to remove 
the effect of the degree $D$ terms. This process is then iterated for
the terms of degrees $d\in\Delta$, in decreasing order,
ultimately obtaining a suitable estimate unless all of the
coefficients $\alpha_{i,d}$ have good rational approximations.

We should comment here on two other approaches to these questions
involving exponential sums. 
Parsell, Prendiville and Wooley \cite{PPW} 
give estimates for general multidimensional sums
based on a multidimensional version of Vinogradov's mean value
theorem. However the bounds obtained save only a small power of $P$
in our notation, whereas our results require a saving in excess of $P^{\cD}$.
Baker \cite[Theorem~5.1]{baker} 
gives a strong result for exponential sums for a
one-variable polynomial, taking account of the Diophantine
approximation properties of all the coefficients. It would be very useful
if such a result were available in our situation.  However, Baker's
proof ultimately depends on estimates for complete exponential sums in
one variable.  Although Baker only requires a relatively weak bound
for such complete sums there appear to be no corresponding estimates
available in the higher-dimensional setting.

Our modified version of Weyl differencing is the subject of
Section~\ref{sec:ES}. We shall apply it in Section \ref{sec:dD} to the
leading forms $F_{1,D}, \dots,F_{r_D,D}$ of degree $D$. The iteration
process is then described in Section \ref{sec:iterate}, producing our 
final bound for the exponential sum $S(\bal)$ in Lemma~\ref{lem:iter}. 
Next, in Section \ref{sec:minor}, we will show how this suffices to
prove \eqref{eq:minor} under the hypothesis in 
the statement of Theorem \ref{thm:main}. To complete the proof of
the theorem
we will establish \eqref{eq:major}
in Section \ref{sec:major}.
We begin with Section \ref{s:geom}, which is 
concerned with 
the facts from algebraic geometry alluded to in the introduction, 
and conclude with Section \ref{sec:crude}, which provides 
the proof of Theorem~\ref{thm:crude}.

\section{Geometric considerations}\label{s:geom}

We commence this section by showing that, 
given any nonsingular system of  forms $\{F_{i,d}\}$, there is 
an equivalent optimal system $\{G_{i,d}\}$.
But an  inspection of the proof of \cite[Lemma~3.1]{bdhb} easily
confirms this fact. Specifically it shows that one can take  
$$
G_{i,d}=F_{i,d}+ \sum_{1\leq k< i}\lambda_k^{(i,d)} F_{k,d} + 
\sum_{\substack{1\leq j\leq n\\  1\leq e< d\\ 1\leq \ell\leq r_e}}
\lambda_{j,\ell,e}^{(i,d)} x_j^{d-e}F_{\ell,e}, $$
for $1\leq i\leq r_d $, $1\leq d\leq D$ and  appropriate integers
$\lambda_k^{(i,d)},\lambda_{j,\ell,e}^{(i,d)}$.

Recall from \eqref{eq:sig-d} that $B_d=\dim(S_d)$, with 
\[S_d=\{\x\in\AA^{n}:\,\rank(J_d(\x))<r_d\}.\]
For an optimal system we can establish the following estimate for
$B_d$, as claimed in \eqref{eq:rhul}.

\begin{lemma}\label{lem:improve}
Suppose that $\{F_{i,d}\}$ is an optimal system of forms.
Let $d\in\Delta$.
Then we have $B_d\le r_d+\dots +r_D-1.$
\end{lemma}

\begin{proof}
In what follows let us write $R_d:=r_d+\dots +r_D$. It will be
convenient to work projectively. Let 
 $d\in \Delta$ and put
\[T_d:=\{[\x]\in\PP^{n-1}:\,\rank(J_d(\x))<r_d\}.\] 
In order to establish the lemma it  suffices to show that $\dim 
T_d\leq R_d-2$.

We introduce the varieties  $V_d, \tilde V_d\subseteq \PP^{n-1}$,
given by 
\[V_d: \quad F_{1,d}=\dots=F_{r_d,d}=0\] 
and 
\[\tilde V_d: \quad F_{2,d}=\dots=F_{r_d,d}=0.\]
Note that only $r_d-1$ forms appear in the definition of $\tilde V_d$.
Since $\{F_{i,d}\}$ is an optimal  system it follows that the varieties
$$
W_d=V_D \cap\dots \cap V_d \quad \mbox{and} \quad
\tilde W_d=V_D\cap\dots \cap V_{d+1}\cap \tilde V_d
$$
are smooth. Note that
$\tilde W_d$ has codimension at most 
$$  r_d-1+r_{d+1}+\dots+r_D=R_d-1  $$ 
in $\PP^{n-1}$, since $r_d\geq 1$.

We are now ready to estimate the dimension of $T_d$.
To do so we note that  $ T_d$ is the set of $[\x]\in\PP^{n-1}$ for
which there exists a point $[\la_1,\dots, \la_{r_d}]\in \PP^{r_d-1}$
such that 
\begin{equation}\label{eq:prop}
\la_1
\nabla F_{1,d}(\x)+\dots +
\la_{r_d}
\nabla F_{r_d,d}(\x)=\mathbf{0}.
\end{equation}
Consider the intersection $I_d= T_d\cap\tilde W_d$.
We claim that $I_d$ is empty.
Any point  $[\x]\in I_d$ for which \eqref{eq:prop} occurs with $\lambda_1\neq 0$
must have $F_{1,d}(\x)=0$, by Euler's identity. But then $[\x]$ must
be a point   in $W_d$ for which the matrix 
\[\left(\begin{array}{c}J_{r_d}(\x)\\ \vdots \\
J_D(\x)\end{array}\right)\]
has rank strictly less than $R_d$. This contradicts the fact that
$W_d$ is smooth. 
Alternatively, any point  $[\x]\in I_d$ for which \eqref{eq:prop}
occurs with $\lambda_1= 0$ 
must produce a singular point on $\tilde W_d$, which is also impossible.
This shows that $I_d$ is empty, whence 
\[ \dim(T_d)<\codim(\tilde W_d)\le R_d-1.\] 
This concludes the proof of the lemma.
\end{proof}

Our  remaining results deal with complete 
intersections. Recall that a variety $X\subseteq \PP^{n-1}$ of
codimension $R$ is said 
to be a complete intersection if its annihilating ideal is generated by $R$
forms. The following result shows that any nonsingular system of forms
produces a  smooth complete intersection of the appropriate degree, 
which is geometrically integral.

\begin{lemma}\label{lem:NS}
Let $\{F_1,\dots, F_R\}$ be a nonsingular system of integral forms,
defining a variety $X$ in $\PP^{n-1}$. Then the annihilating  ideal of
$X$ is generated by $\{F_1,\dots, F_R\}$, 
and $X$ is a smooth complete intersection of codimension
$R$. Moreover, $X$ is geometrically integral and has degree 
$$
\deg(X)=\deg(F_1)\dots \deg(F_R).
$$
\end{lemma}

\begin{proof}
It follows
from
\cite[Exercise~II.8.4]{hart} 
that $X$ is a complete intersection (as a
scheme) of codimension $R$, whose annihilating ideal is generated by
$\{ F_1,\dots, F_R\}$.  The smoothness of $X$ follows from the fact 
that  the system of forms $\{ F_1,\dots, F_R\}$ is nonsingular.

Now the local rings of any smooth scheme are regular. Moreover, a
regular local ring is an integral domain. Thus every local ring of a
smooth scheme must be an integral domain. Moreover, $X$
 is connected by
 \cite[Exercise III.5.5]{hart}. 
 It follows that $X$ is
 geometrically reduced and irreducible, as required. Indeed,
if it failed to be 
 geometrically integral, then it would have  two components 
with a non-empty intersection, since $X$ is connected. But this  is impossible  
since the local ring of any point lying in the intersection would not be an 
integral domain.

Let $d_i=\deg F_i$, for $1\leq i\leq R$. Since
 $X$ is a complete intersection of codimension $R$ in $\PP^{n-1}$, the
 degree of $X$ can be computed via its Hilbert polynomial. Now
 $\{F_1,\dots,F_R\}$ forms a ``regular sequence'' of homogeneous
 elements of $\QQ[\x]$, since 
 $X$ is a complete intersection .
According to
 Harris \cite[Example~13.16]{harris}, the Koszul complex associated to
 the regular sequence $\{F_1,\dots, F_R\}$ is a free 
resolution of the coordinate ring $\QQ[\x]/(F_1,\dots,F_R)$. This
enables us to compute the Hilbert polynomial of $X$ 
and we find that it has  $d_1\dots d_R/(n+1-R)!$
for its leading coefficient.
 Hence
  $\deg(X)=d_1\dots d_R$, as claimed. 
  \end{proof}

Our final result in this section shows that any complete intersection
which is globally defined over $\QQ$ is cut out by integral forms.

\begin{lemma}\label{lem:new}
Let  $X$ be a smooth complete intersection of codimension $R$ which is
globally defined over $\QQ$.  
Then there exist forms $F_1,\dots,F_R$, with  coefficients in $\ZZ$, such that 
the annihilating ideal of $X$  is generated by
$\{ F_1,\dots, F_R\}$. 
\end{lemma}

\begin{proof}
Suppose that $X\subset \PP^{n-1}$ is defined by a system of $R$ equations \eqref{sys}.
We claim that there exist forms $G_i\in \QQ[x_1,\dots,x_n]$ such that 
$\deg(F_i)=\deg(G_i)$, for $1\leq i\leq R$, and such that the annihilating ideal of $X$  is generated by
$\{ G_1,\dots, G_R\}$.  This will establish the lemma on rescaling the forms appropriately.

Let $\deg(F_i)=d_i$ for  $d_1\le\dots \le d_R$. The annihilating ideal of $X$ is 
$\ann(X) :=\langle F_1,\dots ,F_R\rangle$.  
We will argue by induction, the claim being obvious in the case $R=1$ of hypersurfaces. We suppose 
that 
we have found $G_1,\dots,G_r$ with $\ann(X)=\langle G_1,\dots,G_r,F_{r+1},\dots,F_R\rangle.$  Since $X$ is defined over $\QQ$ and 
$F_{r+1}\in \ann(X)$ we must have $F_{r+1}^\sigma \in \ann(X)$
for every  $\sigma\in \mathrm{Gal}(\bar\QQ/\QQ)$.
Thus 
$$
F_{r+1}^\sigma \in \langle G_1,...,G_r,F_{r+1},...,F_R\rangle
$$
for any $\sigma$,
whence
$$
\mathrm{Tr}_{K/\QQ}(cF_{r+1}) \in \langle G_1,...,G_r,F_{r+1},...,F_R\rangle
$$
for any $c \in \bar\QQ$,  where $K$ is the field of definition of $cF_{r+1}$.  
We choose $c$ such that 
$\mathrm{Tr}_{K/\QQ}(cF_{r+1})$ is non-zero and call it 
$G_{r+1}$, so that it has the correct degree.  Thus there exists forms $H_i$ defined over $\bar\QQ$ and constants $e_i\in \bar\QQ$ such that 
\begin{equation}\label{eq:anna}
G_{r+1}=G_1H_1+\dots+G_rH_r+ \sum_{\substack{i}
} e_iF_i,
\end{equation}
where the sum is only for those $i$ 
for which 
$r+1\leq i\leq R$ and 
$d_i=d_{r+1}$.  If there is any choice of $c$ for which one of the $e_i$ is non-zero we can use \eqref{eq:anna} to swap $G_{r+1}$ for the corresponding $F_i$ in the basis 
$\langle G_1,\dots,G_r,F_{r+1},\dots,F_R\rangle$ of $\ann(X)$, thereby completing the induction step.  Alternatively,  if we just have $G_{r+1} \in \langle G_1,\dots ,G_r\rangle$, irrespective of the choice of $c$, then $F_{r+1} \in \langle G_1,\dots,G_r\rangle$, which is impossible.
\end{proof}

\section{Exponential sums}\label{sec:ES}

In this section we consider a quite general situation, independent of
the setup described in Section \ref{s:2}.
Let \[f(x_1,\dots,x_n),g(x_1,\dots,x_n)\in\RR[x_1,\dots,x_n] \]
be polynomials, and let $P\ge 1$ be given.  Suppose that $f$ has
degree at most $d$, and let $F$ be the leading form of degree $d$.
(We shall 
not rule out the possibility that $F$ vanishes identically.) 
We write $F(\x_1,\ldots,\x_d)$ for the $d$-linear polar form, and we 
put $F(\x_1,\dots,\x_d)=\un{F}(\x_1,\dots,\x_{d-1}).\x_d$ 
in analogy to \eqref{dot}. We then take $F^{(i)}$ to be the 
$i$-th component of the row vector 
$\underline{F}(\x_1,\dots,\x_{d-1})$. 

Suppose also that $g$ takes the shape 
\[g=q^{-1}g_1+g_2, \quad   \mbox{with $q\in\NN$ and $
g_1\in\ZZ[x_1,\dots,x_n]$,}
\]
where $g_2$ is a polynomial over $\RR$ satisfying
\beql{partials}
\frac{\partial^{i_1+\dots+i_n}}{\partial^{i_1}x_1\dots\partial^{i_n}x_n}
\,g_2(x_1,\dots,x_n)\ll_{i_i,\dots,i_n}\phi P^{-i_1-\dots-i_n},
\eeq
for some parameter $\phi\ge 1$, uniformly on 
$[-P,P]^n$. 

We give ourselves an  $n$-dimensional box  
$\mathcal{B}\,'\subseteq [-P,P]^n$, with 
sides aligned to the coordinate axes.
We then proceed to consider the exponential sum
\[\Sigma:=\sum_{\x\in \mathcal{B}\,'} e(f(\x)+g(\x)),\]
in which $f$ is the polynomial which mainly concerns us, and $g$ is
regarded as an inconvenient perturbation.
Our estimate for $\Sigma$ will be expressed in terms of
the number $L\ll 1$ defined by 
\[|\Sigma|=P^nL.\]
We now proceed to establish the following bound.

\begin{lemma}\label{lem:final2}
Let $K\geq 1$. Then we have $$
L^{2^{d-1}}\ll P^{-(d-1)n}(q\phi K)^{(d-1)n}(\log P)^n\cM,
$$
where
 $\cM$ counts $(d-1)$-tuples of integer vectors
$(\x_1,\dots,\x_{d-1})$ satisfying
\[|\x_i|<\frac{P}{q\phi K}, \quad (1\le i\le d-1),\]
such that
\[
\|qF^{(i)}(\x_1,\dots,\x_{d-1})\|\le
\frac{1}{P(q\phi)^{d-2}K^{d-1}}, \quad  (1\le i\le n).\]
\end{lemma}

Notice that $\cM\ge 1$ since the $(d-1)$-tuple
$(\mathbf{0},\dots,\mathbf{0})$ is always counted.  The conclusion of
the lemma is therefore trivial unless
$$
q\phi\le P,
$$
as we henceforth suppose.

We start our argument by using $d-2$ standard Weyl differencing steps, to give
\beql{stW}
L^{2^{d-2}}\ll P^{-(d-1)n}\sum_{|\x_1|<P}\dots\sum_{|\x_{d-2}|<P}
\left|\sum_{\x\in I}\psi(\x)\right|,
\eeq
with
\[\psi(\x)=e\left(\Delta_{\x_1,\dots,\x_{d-2}}(f+g)(\x)\right),\]
and where 
$I\subseteq [-P,P]^n$ 
is a box with sides parallel to the
coordinate axes, depending on $\x_1,\dots,\x_{d-2}$. Here
$\Delta_{\x_1,\dots,\x_{d-2}}$ is the usual forward-difference operator.
Normally, since
$f$ potentially has degree $d$, one would want to perform $d-1$ Weyl
differencing steps. 
However we will modify the final step in a way suggested by the van der
Corput argument, and by its $q$-analogue. This will enable us to
eliminate the effect of the polynomial $g$.

We now set
\beql{H}
H=\left[\frac{P}{q\phi}\right],
\eeq
whence $qH\le P/\phi\le P$.  We then have
\[\sum_{\x\in I}\psi(\x)=\sum_{\x\in \ZZ^n}\psi(\x)\chi_I(\x)\]
where $\chi_I$ is the indicator function for $I$, and hence
\begin{align*}
H^n\sum_{\x\in I}\psi(\x)&=
\sum_{\;1\b{u}\le H\;}\sum_{\x\in\ZZ^n}\psi(\x+q\b{u})\chi_I(\x+q\b{u})\\
&=\sum_{\;|\x|\le 2P\;}
\sum_{1\le\b{u}\le H}\psi(\x+q\b{u})\chi_I(\x+q\b{u}),
\end{align*}
where the notation $1\le\b{u}\le H$ is short for 
$1\le u_1,\ldots,u_n\le H$.
Here we have used the fact that $qH\le P$ in order to bound
$|\x|$. Cauchy's inequality now yields 
\begin{align*}
H^{2n}&\left|\sum_{\x\in I}\psi(\x)\right|^2\\
&\ll
P^n\sum_{|\x|\le 2P 
}\left|\sum_{1\le\b{u}\le H}
\psi(\x+q\b{u})\chi_I(\x+q\b{u})\right|^2\\
&=P^n\sum_{1\le\b{u},\b{v}\le H}\;
\sum_{\x\in\ZZ^n}\psi(\x+q\b{v})\chi_I(\x+q\b{v})
\overline{\psi(\x+q\b{u})}\chi_I(\x+q\b{u})\\
&= P^n\sum_{|\b{w}|<H}n(\w)
\sum_{\y\in\ZZ^n}\psi(\y+q\b{w})\chi_I(\y+q\b{w})
\overline{\psi(\y)}\chi_I(\y),
\end{align*}
where
\[n(\w)=\#\{(\u,\v)\in \ZZ^n\cap (0,H]^{2n}:\,\w=\v-\u\}\le H^n.\]
We therefore deduce that
\begin{align*}
\left|\sum_{\x\in I}\psi(\x)\right|^2&\ll P^nH^{-n}
\sum_{|\b{w}|<H}\left|
\sum_{\y\in I'}\psi(\y+q\b{w})\overline{\psi(\y)}\right|\\
&\ll q^n\phi^n\sum_{|\b{w}|<H}\left|
\sum_{\y\in I'}\psi(\y+q\b{w})\overline{\psi(\y)}\right|,
\end{align*}
with some new box $I'\subseteq I\subseteq [-P,P]^n$.
On applying
Cauchy's inequality to (\ref{stW}) we  thus find that 
\beql{stW+}
L^{2^{d-1}}\ll P^{-dn}q^n\phi^n\sum_{|\x_1|<P}\dots\sum_{|\x_{d-2}|<P}
\sum_{|\b{w}|<H}\left|
\sum_{\y\in I'}\psi(\y+q\b{w})\overline{\psi(\y)}\right|.
\eeq

Referring to the definition of the function $\psi$ we see that
\[\psi(\y+q\b{w})\overline{\psi(\y)}
=e\left(\Delta_{\x_1,\dots,\x_{d-2},q\w}(f+g)(\y)\right).\]
Since $f$ is a polynomial of degree $d$, with leading form $F$, 
we see that 
\[\Delta_{\x_1,\dots,\x_{d-2},q\w}(f)(\y)\] 
is a linear polynomial in $\y$, with leading homogeneous part 
\[F(\x_1,\dots,\x_{d-2},q\w,\y) 
=qF(\x_1,\dots,\x_{d-2},\w,\y),\]
where $F(\x_1,\dots,\x_{d})$ is the polar form for $F$, described above. 
Moreover
\[\Delta_{\x_1,\dots,\x_{d-2},q\w}(g_1)(\y)\]
will be an integral polynomial identically divisible by $q$, so that
\[e\left(\Delta_{\x_1,\dots,\x_{d-2},q\w}(q^{-1}g_1)(\y)\right)=1\]
for every $\y\in\ZZ^n$. Finally we consider the exponential factor
involving $g_2$. Using \eqref{partials}, for any non-negative integer
$k$ each of the $k$-th order partial derivatives of
\[\Delta_{\x_1,\dots,\x_{d-2},q\w}(g_2)(\y)\]
will be
\[\ll_k\left(\prod_{i=1}^{d-2}|\x_i|\right)q|\w|\phi P^{-(d-1)-k}
\ll_k qH\phi P^{-1-k}\ll_k P^{-k}\]
for $\y\in I'$, in view of our choice \eqref{H} of $H$.  We may
therefore remove the exponential factor involving $g_2$, using
multi-dimensional partial summation, so as to produce 
\beql{PS}
\sum_{\y\in I'}\psi(\y+q\b{w})\overline{\psi(\y)}\ll
\left|\sum_{\y\in I''}e\big(qF(\x_1,\dots,\x_{d-2},\w,\y)\big)\right|
\eeq
for a further box $I''$. (To be precise, partial summation produces a
bound involving sums over various boxes, and we take $I''$ to be the
box for which the sum is maximal.)

We proceed to sum over $\y$ to get
\[\sum_{\y\in I''}
e\big(qF(\x_1,\dots,\x_{d-2},\w,\y)\big)\ll E\]
with 
\[E=\prod_{i=1}^n\frac{P}{1+P\|qF^{(i)}(\x_1,\dots,\x_{d-2},\w)\|}.\]
Combining the above estimate with \eqref{stW+} and \eqref{PS} leads 
to the bound 
\[L^{2^{d-1}}\ll P^{-dn}q^n\phi^n
\sum_{|\x_1|<P}\dots\sum_{|\x_{d-2}|<P}\sum_{|\b{w}|<H}E.\]

We now follow the strategy used by Davenport in his proof of
\cite[Lemma~3.2]{Dav32}. We write, temporarily,
$\{\theta\}=\theta-[\theta]$ for any real $\theta$, and define
$N(\x_1,\dots,\x_{d-2};\b{r})$ as the number of integer vectors $\w$
for which $|\w|< H$ and
\[\big\{qF^{(i)}(\x_1,\dots,\x_{d-2},\w)\big\}\in
(r_i/P,(1+r_i)/P]\;\; \mbox{for}\;1\le i\le n.\]
We also write $n(\x_1,\dots,\x_{d-2})$ similarly for the number of
integer vectors $\w$ for which $|\w|< H$ and
\[\|qF^{(i)}(\x_1,\dots,\x_{d-2},\w)\|\le P^{-1}\;\;
\mbox{for}\;1\le i\le n.\]
Now if $\w_1,\w_2$ are counted by
$N(\x_1,\dots,\x_{d-2};\b{r})$ then $\w_2-\w_1$ is counted by
$n(\x_1,\dots,\x_{d-2})$, whence
$N(\x_1,\dots,\x_{d-2};\b{r})\le n(\x_1,\dots,\x_{d-2})$
for any $\b{r}\in\RR^n$. Thus
\begin{align*}
\lefteqn{\sum_{|\x_1|<P}\dots\sum_{|\x_{d-2}|<P}\sum_{|\b{w}|<H}
\prod_{i=1}^n\big(1+P\|qF^{(i)}(\x_1,\dots,\x_{d-2},\w)\|\big)^{-1}}\\
&\ll\sum_{\substack{\b{r}\in\ZZ^n\\|\b{r}|\le P}}\prod_{i=1}^n(1+|r_i|)^{-1}
\sum_{|\x_1|<P}\dots\sum_{|\x_{d-2}|<P}N(\x_1,\dots,\x_{d-2};\b{r})\\
&\ll\sum_{\substack{\b{r}\in\ZZ^n\\|\b{r}|\le P}}\prod_{i=1}^n(1+|r_i|)^{-1}
\sum_{|\x_1|<P}\dots\sum_{|\x_{d-2}|<P}n(\x_1,\dots,\x_{d-2})\\
&\ll(\log P)^n\sum_{|\x_1|<P}\dots\sum_{|\x_{d-2}|<
  P}n(\x_1,\dots,\x_{d-2}).
\end{align*}
We therefore conclude that
\beql{near}
L^{2^{d-1}}\ll P^{-(d-1)n}q^n\phi^n(\log P)^n\cN,
\eeq
where $\cN$ counts $(d-1)$-tuples of integer vectors
$(\x_1,\dots,\x_{d-2},\w)$ satisfying
\[|\x_i|<P,\;(1\le i\le d-2)\;\;\;\mbox{and}\;\;\; |\w|<H,\]
such that
\[\|qF^{(i)}(\x_1,\dots,\x_{d-2},\w)\|\le P^{-1}\;\; \mbox{for}\;1\le i\le n.
\]

To estimate $\cN$ we apply the following result, which is Lemma 3.3 of
Davenport \cite{Dav32}.
\begin{lemma}\label{shrink}
Let $L\in M_n(\RR)$ be a real symmetric $n\times n$ matrix.  Let $a>1$
and let
\[N(Z):=\#\{\b{u}\in\ZZ^n: |\b{u}|<aZ,\; \|(L\b{u})_i\|<a^{-1}Z\; \forall
  i\le n\}.\]
Then, if $0<Z_1\le Z_2\le 1$, we have
\[N(Z_2)\ll\left(\frac{Z_2}{Z_1}\right)^nN(Z_1).\]
\end{lemma}
We proceed to choose a parameter $K\ge 1$, as in Lemma
\ref{lem:final2}. It follows in particular that $q\phi K\geq1$, since
$q$ and $\phi$ are at least $1$. We then apply  Lemma \ref{shrink} to
each of the 
vectors $\x_1,\dots,\x_{d-2}$ in succession.  At the $i$-th step we
use 
\[a=P(q\phi K)^{(i-1)/2}, \;\;\; Z_1=(q\phi K)^{-(i+1)/2}\;\;\; 
\mbox{and}\;\;\; Z_2=(q\phi K)^{-(i-1)/2}.\]
Finally we apply Lemma \ref{shrink} to $\w$ with 
\[a=(HP)^{1/2}(q\phi K)^{(d-2)/2}, \]
and
\[Z_1=K^{-1}H^{1/2}P^{-1/2}(q\phi K)^{-(d-2)/2},\;\;\;
 Z_2=H^{1/2}P^{-1/2}(q\phi K)^{-(d-2)/2}.\]
One readily verifies that these choices satisfy the conditions for the
lemma, and concludes that
\[\cN\ll (q\phi)^{(d-2)n}K^{(d-1)n}\cM,\]
where $\cM$ is as in the statement of Lemma \ref{lem:final2}.
The required estimate then follows on inserting this into  \eqref{near}.

\section{The degree $D$ case}\label{sec:dD}

We now return to the situation in Section  \ref{s:2}.  Suppose that we have a
parameter $\alpha_{i,d}\in\RR$ corresponding to each form $F_{i,d}$,
for $1\le i\le r_d$ and each $1\leq d\leq D$.
Recall that a box $\mathcal{B}\subseteq [-1,1]^n$, a modulus $M\in\NN$
and an integer vector $\m_0$ are
given, and are fixed once for all.

We apply the work of the previous section with
\[f(\x)=\sum_{j=1}^D\sum_{i=1}^{r_j}\alpha_{i,j}F_{i,j}
(M\x+\m_0)\] 
and $g(\x)=0$. If we take 
\[\mathcal{B}\,'=\{\x:\,M\x+\mathbf{m}_0\in P\mathcal{B}\}\] 
then $\mathcal{B}\,'\subseteq [-P,P]^n$ for large enough $P$ (since 
$\mathbf{m}_0=\mathbf{0}$ for $M=1$). 
We may set $q=1$ and $\phi=1$ in the notation of 
Section~\ref{sec:ES}. Moreover the leading 
form of $f$ has degree $D$  and is given by
\beql{Fd}
F(\x)=M^D\sum_{i=1}^{\rho}\alpha_{i,D}F_{i,D}(\x),
\eeq
where we have written
\[r_D=\rho\]
for brevity. Our problem now corresponds closely to that
encountered by Birch \cite{birch}, and we shall follow his line of
attack. The outcome will be that either the exponential sum is small,
or the coefficients $\alpha_{i,D}$ are all close to rationals with a small
denominator.  This denominator will be denoted by $q$, and is not to
be confused with the number $q=1$ above, which is related to the
polynomial $g(\x)=0$.

The analysis of the previous section shows that we have a bound of the
shape in Lemma \ref{lem:final2}, in which the parameter $K$ is at our
disposal. We will take $K=\max\{1,K_1\}$, with
\[K_1=P\left(\frac{L^{2^{D-1}}}{(\log P)^{n+1}}\right)^{1/(n-B_D)},\]
where $B_D$ is given by \eqref{eq:sig-d}.  The reader should observe
that it is perfectly permissible to use a value for $K$ which depends
on $L$.
We now examine $\cM$, considering three different cases.  The first of
these is that in which $K_1\le 1$, so that 
\[L^{2^{D-1}}\le P^{B_D-n}(\log P)^{n+1}.\]
This is satisfactory for our purposes (see Lemma \ref{lD}). 
We will therefore assume henceforth that $K=K_1>1$.  

The second case is that in which all the
$(D-1)$-tuples counted by $\cM$ correspond to elements of the set
$\widehat{S}_D$ given by (\ref{eq:hat-S}). In this situation we will
apply the following estimate.
\begin{lemma}\label{L1}
Let $d\leq D$, let $P\ge 1$ and let $\cM_0(P)$ be the number of
$(d-1)$-tuples of 
 vectors $(\x_1,\dots,\x_{d-1})\in \widehat{S}_d(\ZZ)$ having $\max|\x_i|\le P$.
Then
$$
\cM_0(P)\ll P^{B_d+n(d-2)}.
$$
\end{lemma}

\begin{proof}
Since $S_d$ is the intersection of $\widehat{S}_d$ with the diagonal
\[{\rm Diag}=\{(\x,\dots,\x)\in(\AA^{n})^{d-1}\},\]
we see that
\begin{align*}
\dim(\widehat{S}_d)&\le B_d+{\rm codim}({\rm Diag})\\
&= B_d +n(d-2).
\end{align*}
We now apply Lemma 3.1 of Birch  \cite{birch} to conclude the proof. 
\end{proof}

Now, with the above notation, one has 
\[\cM\le\cM_0(P/K)\ll
\left(\frac{P}{K}\right)^{B_D+n(D-2)}.\]
In this case Lemma \ref{lem:final2} yields
\[L^{2^{D-1}}\ll \left(\frac{K}{P}\right)^{n-B_D}(\log P)^n,
\]
Since $K=K_1$ we deduce that
\[L^{2^{D-1}}\ll L^{2^{D-1}}(\log P)^{-1}.\]
Thus this second case cannot occur if $P$ is sufficiently large.

This takes us to the third case, in which $K=K_1>1$ and there is some
$(D-1)$-tuple counted by $\cM$ for which 
\[\rank(\widehat{J}_D(\x_1,\dots,\x_{D-1}))=r_D=\rho.\]
Suppose the matrix corresponding to columns $j_1,\dots,j_{\rho}$ has
non-zero determinant. Calling the matrix $W$, we have
\[W_{ik}=F_{i,D}^{(j_k)}(\x_1,\dots,\x_{D-1}), \quad (1\le i,k\le \rho),\] 
where 
$F_{i,D}^{(j_k)}(\x_1,\dots,\x_{D-1})$ is the $j_k$-th component of the row vector 
$\underline{F}_{i,D}(\x_1,\dots,\x_{D-1})$.  But then 
 \eqref{Fd} yields
\begin{align*}
F^{(j_k)}(\x_1,\dots,\x_{D-1})
&=M^D\sum_{i=1}^{\rho}\alpha_{i,D}
F_{i,D}^{(j_k)}(\x_1,\dots,\x_{D-1})\\
&=M^D\sum_{j=1}^{\rho}\alpha_{i,D}W_{ik}.
\end{align*}
We record for future reference the fact that
\beql{Wb}
H(W)\ll (\max|\x_h|)^{D-1}\ll
\left(\frac{P}{K_1}\right)^{D-1},
\eeq
where we use $H(W)$ to denote the maximum of $|W_{jk}|$. 

Since $(\x_1,\dots,\x_{D-1})$ is counted by $\cM$ it follows that
\[\left\|M^D\sum_{i=1}^{\rho}\alpha_{i,D}W_{ik}\right\|\le
\frac{1}{PK_1^{D-1}}, \quad  (1\le k\le \rho).\]
We therefore write
\beql{sys+}
M^D\sum_{i=1}^{\rho}\alpha_{i,D}W_{ik}=n_k+\xi_k
\eeq
for $k=1,\dots,\rho$, with $n_k\in\ZZ$ and \[|\xi_k|\le\frac{1}{PK_1^{D-1}}.\]
We proceed to abbreviate the system \eqref{sys+} by writing
\[M^DW\un{\alpha}=\un{n}+\un{\xi},\]
and then multiply by the adjoint, $W'$ say, of $W$ to see that
\[M^D\det(W)\un{\alpha}=W'\un{n}+W'\un{\xi}.\]
However $W'$ is an integer matrix, with
\[H(W')\ll H(W)^{\rho-1}\ll
\left(\frac{P}{K_1}\right)^{(D-1)(\rho-1)},\]
by \eqref{Wb}.  It follows that
\[\|M^D\det(W)\alpha_{i,D}\|\ll\left(\frac{P}{K_1}\right)^{(D-1)(\rho-1)}
\frac{1}{PK_1^{D-1}},\]
for $i=1,\dots,\rho$.  If we now write $q=M^D|\det(W)|\ll H(W)^\rho$,
then $q$ will be a positive integer, since we chose $W$ to have
non-zero determinant. Moreover for large enough $P$ we will have $q\le Q$, where
\[Q=\left(\frac{P}{K_1}\right)^{(D-1)\rho}\log P,\] 
and 
\[\|q\alpha_{i,D}\|\le QP^{-D}.\]
We may now summarize all these conclusions as follows.

\begin{lemma}\label{lD}
Let $|S(\bal)|=P^nL$ and write $\rho=r_D$.  Then if $P$ is large
enough, either
\beql{Lb}
L^{2^{D-1}}\le P^{B_D-n}(\log P)^{n+1},
\eeq
or there is a $q\le Q$, with 
\[Q\le\left((\log P)^{n+1}L^{-2^{D-1}}\right)^{(D-1)\rho/(n-B_D)}\log P,\]
such that
\[\|q\alpha_{i,D}\|\le QP^{-D}, \quad (1\leq i\le \rho).\] 
\end{lemma}

We now ask what one can say about the minor arc integral using Lemma
\ref{lD}.  For any $L_0>0$ we write $\cA(L_0)$
for the set of $R$-tuples of
values $\alpha_{i,d}$ with $d\le D$, $i\le r_d$ such that $L_0< L\le
2L_0$. Then if $L_0$ is such that \eqref{Lb} holds, the contribution
to the minor arc
integral will be \[\ll P^{n+\ve-(n-B_D)/2^{D-1}},\]
for any fixed $\ve>0$.  This is satisfactory if $(n-B_D)/2^{D-1}>\cD$,
or in other words, if $n>B_D+2^{D-1}\cD$.

In the alternative case we see that there is an
integer $q\leq Q$ such that every $\alpha_{i,D}$, for $1\le i\le
\rho$, has an approximation 
\[\alpha_{i,D}=a_{i,D}/q+O(QP^{-D}q^{-1})\] 
with $a_{i,D}\in\ZZ$ and $0\le a_{i,D}\le q$.  
Hence
\begin{align*}
\meas(\cA(L_0))&\ll\sum_{q\le Q} q^{\rho}(QP^{-D}q^{-1})^\rho\\ 
&\ll Q^{1+\rho}P^{-D\rho}\\
&\ll L_0^{-2^{D-1}(D-1)\rho(1+\rho)/(n-B_D)}P^{\ve-D\rho}.
\end{align*}
The corresponding contribution to the minor arc integral will
therefore be
\[\ll L_0^{1-2^{D-1}(D-1)\rho(1+\rho)/(n-B_D)}P^{n+\ve-D\rho}.\]

Hence, for example, if our system has forms of degree $D$ only, then
$\cD=D\rho$ and we have a satisfactory bound when
\[n>B_D+\rho(\rho+1)2^{D-1}(D-1),\]
providing that $L_0^{-1}$ exceeds some small fixed power of $P$. This
corresponds precisely to the condition on $n$ in \eqref{br}.

\section{Exponential sums --- the iterative argument}
\label{sec:iterate}

In the previous section we showed that either $S(\bal)$ (or
equivalently $L$) is small, as expressed by \eqref{Lb}, or that the
coefficients $\alpha_{i,D}$ all have good rational approximations with
the same small denominator $q$.  In this section we iterate this idea,
assuming that we have good approximations for $\alpha_{i,j}$ for
$d<j\le D$ and $1\leq i\le r_j$, and deducing either that $L$ is small, or that
the values $\alpha_{i,d}$ also have good rational approximations for
$1\leq i\le r_d$.

Thus we suppose we have a degree $d<D$ in $\Delta$,
and we suppose that there is a
positive integer $q\le Q$ such that
\[\|q\alpha_{i,j}\|\le QP^{-j}
\;\;\;\mbox{for}\;\;\; d<j\le D\;\mbox{and}\;1\le i\le r_j.\]
We then define
\[f(\x)=\sum_{j=1}^d\sum_{i=1}^{r_j}\alpha_{i,j}
F_{i,j}(M\x+\m_0)\] 
and
\[g(\x)=\sum_{j=d+1}^D\sum_{i=1}^{r_j}\alpha_{i,j}F_{i,j}
(M\x+\m_0),\] 
and we write
\[r_d=\rho\]
for brevity.
Then the polynomial $f$ has degree 
at most 
$d$ and the
leading form of degree $d$ is now
\[F(\x)=M^d\sum_{i=1}^{\rho}\alpha_{i,d}F_{i,d}(\x).\]
We also write
\[\alpha_{i,j}=\frac{a_{i,j}}{q}+\theta_{i,j}
\;\;\;\mbox{for}\;\;\; d<j\le D\;\mbox{and}\;1\le i\le r_j,\]
so that \[|\theta_{i,j}|\le QP^{-j}q^{-1}.\] To complete the setup we put
\[g_1(\x)=\sum_{j=d+1}^D\sum_{i=1}^{r_j}a_{i,j}
F_{i,j}(M\x+\m_0)\] 
and
\[g_2(\x)=\sum_{j=d+1}^D\sum_{i=1}^{r_j}\theta_{i,j}
F_{i,j}(M\x+\m_0).\]
Then $g=q^{-1}g_1+g_2$ is in the required shape to apply the work of
Section \ref{sec:ES}, and in particular we
see that \eqref{partials} holds with $\phi=Q/q$.

We now proceed exactly as in the previous section, taking
$K=\max\{1,K_1\}$ with
\[K_1=PQ^{-1}\left(\frac{L^{2^{d-1}}}{(\log P)^{n+1}}\right)^{1/(n-B_d)}.\]
Then, if $K_1\le 1$ as in the first case of the argument, we will have
\[L^{2^{d-1}}\le (P/Q)^{B_d-n}(\log P)^{n+1},\]
which will be satisfactory.  The second case will be that in which
all the $(d-1)$-tuples counted by $\cM$ correspond to elements of the set
$\widehat{S}_d$. Since $q\phi=Q$ we then have
\[\cM\le\cM_0(P/QK)\ll
\left(\frac{P}{QK}\right)^{B_d+n(d-2)}\]
by Lemma \ref{L1}, after which Lemma \ref{lem:final2} yields
\[L^{2^{d-1}}\ll\left(\frac{QK}{P}\right)^{n-B_d}(\log P)^n.\]
Since $K=K_1$ we deduce that
\[L^{2^{d-1}}\ll L^{2^{d-1}}(\log P)^{-1},\]
and as before we conclude that this second case cannot occur if $P$ is
sufficiently large. 

The third case is that in which some $(d-1)$-tuple counted by $\cM$ has
\[\rank(\widehat{J}_d(\x_1,\dots,\x_{d-1}))=r_d=\rho.\]
Here the argument again follows that in the previous section, but now
\[H(W)\ll (\max|\x_h|)^{d-1}\ll\left(\frac{P}{QK_1}\right)^{d-1},\]
and
\[\left\|qM^d\sum_{i=1}^{\rho}\alpha_{i,d}W_{ik}\right\|\le
\frac{1}{PQ^{d-2}K_1^{d-1}}, \quad  (1\le k\le \rho).\] This time we put
\[qM^d\sum_{i=1}^{\rho}\alpha_{i,d}W_{ik}=n_k+\xi_k\] with $n_k\in \ZZ$ and
\[|\xi_k|\le\frac{1}{PQ^{d-2}K_1^{d-1}}.\] We will then have
\[H(W')\ll H(W)^{\rho-1}\ll
\left(\frac{P}{QK_1}\right)^{(d-1)(\rho-1)},\]
whence
\[\|qM^d\det(W)\alpha_{i,d}\|\ll\left(\frac{P}{QK_1}\right)^{(d-1)(\rho-1)}
\frac{1}{PQ^{d-2}K_1^{d-1}}\]
for $i=1,\dots,\rho$.  We therefore set $q^*=M^d|\det(W)|\ll
H(W)^{\rho}$, so that $q^*\le Q^*$ with 
\[Q^*=\left(\frac{P}{QK_1}\right)^{(d-1)\rho}\log P,\]
and \[\|qq^*\alpha_{i,d}\|\le QQ^*P^{-d}.\]
We may now summarize all these conclusions as follows.
\begin{lemma}\label{ld}
Let $|S(\bal)|=P^nL$. Suppose that $d\in\Delta$ and that
\[\|q\alpha_{i,j}\|\le QP^{-j}, \quad \mbox{for $d<j\le D$ and $1\le i\le r_j$}\]
with $q\le Q$.  Then if $P$ is large enough, either 
\[L^{2^{d-1}}\le\left(\frac{Q}{P}\right)^{n-B_d}(\log P)^{n+1},\] 
or there is a $q^*\le Q^*$, with 
\[Q^*=\left((\log P)^{n+1}L^{-2^{d-1}}\right)^{(d-1)r_d/(n-B_d)}\log
P,\] 
such that
$$
\|qq^*\alpha_{i,d}\|\le QQ^*P^{-d}, \quad  (1\leq i\leq r_d).
$$
\end{lemma}

We may of course interpret Lemma~\ref{lD} as a special case of
Lemma~\ref{ld}, corresponding to $d=D$ and $Q=1$.  

Our plan is to use Lemma \ref{lD} followed by repeated applications
of Lemma \ref{ld} for the successively smaller values of
$d\in\Delta$. Thus in Lemma~\ref{lD} either 
\[L^{2^{D-1}}\le P^{B_D-n}(\log P)^{n+1},\]
or there is a $q_D\le Q_D$, with 
\[Q_D= \left((\log P)^{n+1}L^{-2^{D-1}}\right)^{(D-1)r_D/(n-B_D)}\log P,\]
such that
\[\|q_D\alpha_{i,D}\|\le Q_DP^{-D}, \quad (1\leq i\le r_D).\]
If the second case holds we may
then apply Lemma \ref{ld} for degree
\[D':=\max\{d\in\Delta:\,d<D\}.\]
We then deduce either that
\[L^{2^{D'-1}}\le (Q_D/P)^{n-B_{D'}}(\log P)^{n+1},\]
or that there is a $q_{D'}=q_Dq^*\le Q_{D'}=Q_DQ^*$, with \[Q^*= \left((\log
  P)^{n+1}L^{-2^{D'-1}}\right)^{(D'-1)r_{D'}/(n-B_{D'})}\log P,\]
such that
\[\|q_{D'}\alpha_{i,D'}\|\le Q_{D'}P^{-D'}, \quad (1\leq i\le r_{D'}).\]
Continuing in this manner we produce a succession of values $Q_d$ for
decreasing values of $d$ in $\Delta$, taking the form
\beql{Qform}
Q_d=(\log P)^{e(d)}L^{-s_d},\;\;\;(d\in\Delta),
\eeq
where $e(d)$ is some easily computed but unimportant exponent, and
$s_d$ is given by \eqref{eq:def-sj}.

When $1\le j\le D$ but $j\not\in\Delta$ it will be convenient to set
$Q_j=Q_k$, where $k$ is the smallest integer in $\Delta$
with $k>j$. We will also put $Q_{D+1}=1$.
In view of \eqref{eq:def-sj} we have $s_j=s_k$ so that
\eqref{Qform} extends to give
\beql{Qform+}
Q_d=(\log P)^{e(d)}L^{-s_d},\;\;\;(1\le d\le D)
\eeq
for appropriate exponents $e(d)$.  Now, for a general 
exponent $j\in\Delta$, as we iterate we will either obtain a bound
\beql{exp}
L^{2^{j-1}}\le (Q_{j+1}/P)^{n-B_j}(\log P)^{n+1}\;\;\;(j\in\Delta),
\eeq
or we will find a positive integer $q_j$ satisfying \beql{ci}
q_k\mid q_j\;\;\;(k>j,\; k\in\Delta),
\eeq
\beql{cii}
q_j\le Q_j,
\eeq
and
\beql{ciii}
\|q_j\alpha_{i,j}\|\le Q_jP^{-j}, \quad (1\leq i\le r_j).
\eeq
When $1\le j\le D$ but $j\not\in\Delta$ it will also be convenient to set
$q_j=q_k$, where $k$ is the smallest integer in $\Delta$
with $k>j$.  With this convention we then have $q_j\le Q_j$ and
$q_j\mid q_{j+1}$ in general.

We can now partition the $R$-tuples
$\bal$ into sets $I^{(1)}_d$ (for $d\in\Delta$)
and $I^{(2)}$, as follows. The set $I^{(1)}_d$ consists of those $\bal$ for which
(\ref{exp}) fails for $j>d$, but holds for $j=d$.
The set $I^{(2)}$ then consists of the remaining
$R$-tuples $\bal$, for which (\ref{exp}) fails for all
$j\in\Delta$.

It follows from (\ref{Qform}) that if (\ref{exp}) holds one has
\[L^{2^{j-1}+(n-B_j)s_{j+1}}\ll P^{B_j-n+\ve},\]
for any fixed $\ve>0$. We therefore draw the following conclusion.

\begin{lemma}\label{lem:iter} Let $d\in\Delta$ and $\bal\in I^{(1)}_d$.  Then
\beql{LQb}
L^{2^{d-1}+(n-B_d)s_{d+1}}\ll P^{B_d-n+\ve}.
\eeq
Moreover there are positive integers $q_j$ such that
the conditions \eqref{ci}, \eqref{cii} and \eqref{ciii} hold for all
relevant values of $j>d$. 

Similarly, if $\bal\in I^{(2)}$ then there are positive integers $q_j$ such that
the conditions \eqref{ci}, \eqref{cii} and \eqref{ciii} hold for all
values of $j\in\Delta$. 
\end{lemma}

We conclude this section by remarking that it may be possible to
improve on the above estimates in certain cases. The numbers $q_d$ are
built up from a sequence of factors. This would allow one to replace
the argument in Section \ref{sec:ES} by one in which there were
several van der Corput steps, using various factors of $q_d$. In our
present treatment, when one uses Lemma \ref{shrink}, the ratio
$Z_2/Z_1$ is $q\phi K$ for the first $d-2$ steps, and $K$ for the
final step.  In the proposed variant these values become more equal,
which should be to our advantage.  However this can only be of use
when $\Delta$ contains at least three values $d\ge 3$, since the number
$q$ in our argument would need to have at least two factors, and so the
number $d-1$ of squaring steps would have to be at least two.

\section{The minor arc contribution}\label{sec:minor}

As in Section \ref{sec:dD},
for any $L_0>0$ we write $\cA(L_0)$
for the set of $R$-tuples of
values $\alpha_{i,d}$ with $d\le D$, $i\le r_d$ such that $L_0< L\le
2L_0$. We also write
$\cA(L_0;I^{(1)}_d)=I^{(1)}_d\cap\cA(L_0)\cap\minor$, and similarly
for $\cA(L_0;I^{(2)})$. 
In order to establish the required minor arc estimate
\eqref{eq:minor} we begin by examining
\[T(I^{(1)}_d):=\int_{\cA(L_0;I^{(1)}_d)}|S(\bal)| \d \bal\] 
for $d\in \Delta$.

When $d=D$ we have 
$$
L^{2^{D-1}}\le P^{B_{D}-n+\ve},  
$$
by (\ref{LQb}).  
Since $\meas(\cA(L_0;I^{(1)}_D))\le 1$ and $|S(\bal)|=P^nL$ it follows
that 
\[T(I^{(1)}_{D})\ll P^{n} L_0\ll P^{n-(n-B_D)/2^{D-1}+\ve}.\] 
Thus we will have a satisfactory
estimate $T(I^{(1)}_D)\ll P^{n-\mathcal{D}-\delta}$, for some 
$\delta>0$, provided that
\beql{NB1}
\mathcal{D}\frac{2^{D-1}}{n-B_D}<1.
\eeq

We now consider the general case, in which $\bal\in I^{(1)}_d$ for
some $d<D$ in $\Delta$. Thus (\ref{LQb}) holds, so that
\beql{L0b}
L_0^{2^{d-1}/(n-B_d)+s_{d+1}}\ll P^{-1+\ve}.
\eeq
When $d=D$ we estimated $\meas(\cA(L_0;I^{(1)}_D))$ trivially, but when $d<D$ we
have useful information on the numbers $\alpha_{i,j}$ for $j>d$, since
we know that (\ref{ciii}) applies for these.
Thus there are positive integers $q_D, q_{D-1},\dots,q_{d+1}$
depending on $\bal$ and satisfying (\ref{ci}) and (\ref{cii}), such
that 
\[\alpha_{i,j}=\frac{a_{i,j}}{q_j}+O(q_j^{-1}Q_jP^{-j}), \quad (d<
j\le D,\; 1\le i\le r_j),\] 
with $0\le a_{i,j}\le q_j$. Thus, given $q_j$, each individual $\alpha_{i,j}$ takes
values in a set of measure $O(Q_jP^{-j})$, and the $r_j$-tuple
$(\alpha_{1,j},\dots,\alpha_{r_j,j})$ has values in a set of measure
$O\big((Q_jP^{-j})^{r_j}\big)$. At this point we recall our
convention concerning the values of $q_j$ and $Q_j$ when $j\not
\in\Delta$.  With this in mind we see that $q_{d+1}$ determines $O(P^{\ve})$
possibilities for $q_{d+2},\dots,q_D$, by \eqref{ci}, and we conclude
that 
\begin{align*}
\meas(\cA(L_0;I^{(1)}_d))
&\ll P^{\ve}Q_{d+1}\prod_{j=d+1}^D(Q_jP^{-j})^{r_j}.
\end{align*}
Hence, using \eqref{Qform+}, we obtain 
\[\meas(\cA(L_0;I^{(1)}_d)) \ll P^{\ve-(d+1)r_{d+1}-\dots -Dr_D}
L_0^{-(s_{d+1}+s_{d+1}r_{d+1}+\dots+s_Dr_D)}.\]
Recalling the notation \eqref{eq:Dj} for $\mathcal{D}_j$ and that
$|S(\bal)|=P^nL$, it now follows that 
\[T(I^{(1)}_d)\ll P^{n-\mathcal{D}+\mathcal{D}_d+\ve}  
L_0^{1-(s_{d+1}+s_{d+1}r_{d+1}+\dots+s_Dr_D)},\]
with $L_0$ subject to (\ref{L0b}). Thus we will have a satisfactory
estimate $T(I^{(1)}_d)\ll P^{n-\mathcal{D}-\delta}$, for some 
$\delta>0$, provided that
\beql{cd}
\mathcal{D}_d\left(\frac{2^{d-1}}{n-B_d} +s_{d+1}\right) 
+s_{d+1}+\sum_{j=d+1}^D s_jr_j<1.
\eeq
It is clear now that the corresponding condition (\ref{NB1}) for $d=D$
is just a special case of this.

For the integral 
\[T(I^{(2)}):=\int_{\cA(L_0;I^{(2)})}|S(\bal)| \d \bal\] 
we will provide an estimate for $L$ by
using the fact that our $R$-tuple $\bal$ belongs to $\minor$. It
follows from \eqref{ci}, \eqref{cii} and \eqref{ciii} that
\[\|q_1\alpha_{i,d}\|\le q_1q_d^{-1}Q_dP^{-d}\le Q_1Q_dP^{-d}\]
for $1\le d\le D$ and $i\le r_d$. If we write $s_{\rm max}=\max s_d$
and $e_{\rm max}=\max e(d)$, then \eqref{Qform+} yields 
\[\|q_1\alpha_{i,d}\|\le L^{-2s_{\rm max}}P^{-d}(\log P)^{2e_{\rm max}},\]
with \[q_1\le Q_1\le L^{-s_{\rm max}}(\log P)^{e_{\rm max}}.\]
Let $\varpi$ be as in Section \ref{s:2}. Then if
$P$ is large enough it follows that one would have
\[\|q_1\alpha_{i,d}\|\le P^{-d+\varpi}\;\;\;(1\le d\le D,\;\; i\le r_d)\]
with
\[q_1\le P^{\varpi}\]
so long as
\[L\ge P^{-\varpi/4s_{\rm max}}.\]
However this would place $\bal$ in the major arcs, in view of the
definition \eqref{majdef}.  We therefore conclude that
\beql{I2L0}
L_0\ll P^{-\varpi/4s_{\rm max}}
\eeq
for $\bal\in I^{(2)}$.

We can now use \eqref{ci}, \eqref{cii} and \eqref{ciii} as before to show that
\begin{align*}
\meas(\cA(L_0;I^{(2)})) &\ll  P^{\ve}Q_1\prod_{j=1}^D(Q_jP^{-j})^{r_j}\\
&\ll  P^{\ve-r_1-2r_2-\dots -Dr_D}
L_0^{-(s_1+s_1r_1+\dots+s_Dr_D)}.
\end{align*}
It follows that
\[T(I^{(2)})\ll P^{n-\mathcal{D}+\ve} 
L_0^{1-(s_1+s_1r_1+\dots+s_Dr_D)}.\] 
In view of \eqref{I2L0} we obtain a satisfactory bound 
$T(I^{(2)})\ll P^{n-\mathcal{D}-\delta}$, for some $\delta>0$, 
under the constraint
\[s_1+\sum_{j=1}^D s_jr_j<1.\]
The reader should notice that this condition is the case $d=0$ of
\eqref{cd}, since we have defined $\mathcal{D}_0=0$ in connection with
\eqref{eq:Dj}.

We therefore see that we have a satisfactory minor arc estimate
provided that \eqref{cd} holds for all $d\in\Delta\cup\{0\}$, as
required for Theorem~\ref{thm:main}.

\section{The major arc contribution}\label{sec:major}

We now turn
to the major arc analysis, with the goal of establishing
\eqref{eq:major} under suitable hypotheses on $\major$ and the forms
$(F_{i,d})$. Let us define $$
S(\a,q):=\sum_{\x\bmod{q}}
e_q\left(\sum_{d=1}^D\sum_{i=1}^{r_d}a_{i,d} F_{i,d}(M\x+\m_0)\right),
$$
for $\a=(a_{i,d})\in (\ZZ/q\ZZ)^R$ with $\gcd(q,\a)=1$.
Next, define the truncated singular series
$$
\ss(H):=\sum_{q\leq H} \frac{1}{q^n}\sum_{\substack{\a \bmod{q}\\ \gcd(q,\a)=1}}
S(\a,q),
$$
for any $H>0$. We put  $\ss=\lim_{H\rightarrow \infty}\ss(H)$,
whenever this limit exists. 
We will also need to study the integral
\begin{equation}
  \label{21-si}
\mathfrak{I}(H)=
\frac{1}{M^n}\int_{[-H,H]^R}
\int_{\mathcal{B}} e\left(\sum_{d=1}^D\sum_{i=1}^{r_d}\gamma_{i,d} F_{i,d}(\x)\right)
\d\x\d \bga,
\end{equation}
for any $H>0$, where $\bga=(\gamma_{i,d}).$
Recalling \eqref{eq:local-inf}, we have 
$\sigma_\infty=\lim_{H\rightarrow \infty}
\mathfrak{I}(H)$, whenever the limit  exists. The main aim of this
section is to establish the 
following result.

\begin{lemma}\label{lem:major}
Assume that 
\begin{equation}\label{eq:ant}
s_1+\sum_{j=1}^D s_jr_j<1.
\end{equation}
Then the singular series $\ss$ and the singular integral
$\mathfrak{I}$ are absolutely convergent.  Moreover, if we choose 
\beql{eq:choosedelta}
\varpi=\frac{1}{2R+4}
\eeq
then there is a positive constant $\delta$ such that
$$
\int_{\major}S(\bal)\d\bal=\ss \mathfrak{I}P^{n-\mathcal{D}}
+O(P^{n-\mathcal{D}-\delta}).
$$
\end{lemma}

The condition in the lemma is the case $d=0$ of
the condition in the statement
of Theorem \ref{thm:main}.  Once the lemma is established we will 
have $\ss=\prod_p \sigma_p$ by the argument of Davenport 
\cite[Chapter 17]{DavAM}, for example. We leave the details to the reader.
Theorem \ref{thm:main} then follows.  

Recall the definition of the major arcs $\mathfrak{M}$  from Section
\ref{s:2},  defined in terms of the parameter $\varpi\in (0,1/3)$.  
Any $\bal\in \mathfrak{M}_{q,\a}$ can be written
$$
\alpha_{i,d}=\frac{a_{i,d}}{q}+\theta_{i,d}
$$
for $1\leq i\leq r_d$ and $d\in \Delta$.
Our  first step in the analysis of $S(\bal)$ on $\mathfrak{M}$
is an analogue of \cite[Lemma~5.1]{birch}. The 
argument is well-known and we 
leave the details to the reader.  
It leads to the 
conclusion that 
\begin{equation}\label{eq:51}
\begin{split}
S(\bal)&-(qM)^{-n}P^n S(\a,q)J(\bga)\\
&\ll q\sum_{1\leq d\leq D}\sum_{1\leq i\leq r_d} |\theta_{i,d}|P^{n+d-1}+qP^{n-1},
\end{split}
\end{equation}
where $J(\bga)$ is given by \eqref{eq:defn-J}, and $\bga$ is the
vector whose $i,d$ entry is $\theta_{i,d}P^d$. 
But then it follows that 
$$
S(\bal)=(qM)^{-n}P^n S(\a,q)J(\bga)+O(P^{n-1+2\varpi}),
$$
for any $\bal\in \mathfrak{M}$.
The major arcs are easily seen to have measure
$O(P^{-\mathcal{D}+(2R+1)\varpi})$. Hence   
$$
\int_{\mathfrak{M}}S(\bal) \d\bal =P^{n-\mathcal{D}}
\mathfrak{S}(P^\varpi)\mathfrak{I}(P^\varpi) +O(P^{n-\mathcal{D}-1+(2R+1)\varpi+2\varpi}).
$$
This error term is satisfactory for Lemma \ref{lem:major} if $\varpi$
is taken as in \eqref{eq:choosedelta}. 

In order to complete the proof of the lemma, it remains to show that 
$\ss$ and  $\mathfrak{I}$ are absolutely convergent when
\eqref{eq:ant} holds, and  
that there is a positive constant $\delta$ such that 
\begin{equation}\label{eq:show1}
\mathfrak{S}-\mathfrak{S}(H)\ll H^{-\delta}
\end{equation}
and 
\begin{equation}\label{eq:show2}
\mathfrak{I}-\mathfrak{I}(H)\ll H^{-\delta},
\end{equation}
for any $H>0$.

Beginning with the singular series, 
we proceed to use \eqref{eq:51} and our Weyl estimate Lemma 
\ref{lem:iter} to estimate the complete exponential sum $S(\a,q)$, as follows.

\begin{lemma}\label{lem:majorA}
Let $\ve>0$ be given. Then 
$$
S(\a,q)\ll q^{n+\ve} \min_{j\in \Delta} 
\left(\frac{\gcd(q,\a^{(j)},\dots,\a^{(D)})}{q}\right)^{1/s_j},
$$
where $\a^{(j)}=(a_{1,j},\dots,a_{r_j,j})$ for any $j\in \Delta$.
\end{lemma}

\begin{proof}
Noting that $J(\mathbf{0})\gg1$, we may take 
 $(\theta_{i,d})=\mathbf{0}$ in \eqref{eq:51} to conclude that 
 $$
 S(\a,q)\ll  \frac{q^n|S(\bal)|}{P^n} + \frac{q^{n+1}}{P},
$$
with $\bal=a^{-1}\a$.
In what follows we will take $P=q^{A}$ for some large value of $A$ to
be specified during the course of the proof. Assuming that $A>n+1$, in
the first instance, it follows from the previous bound that  
\begin{equation}\label{eq:52}
 S(\a,q)\ll 1+ q^n L,
\end{equation}
where $L$ is defined via $|S(\bal)|=P^nL$.
We now apply Lemma~\ref{lem:iter}.

If there exists $d\in \Delta$ such that $\bal\in I_d^{(1)}$ then 
$L$ satisfies 
\eqref{LQb}.  Once combined  with \eqref{eq:52}, this gives
$$
 S(\a,q)\ll 1+ \frac{q^nP^\ve   }{P^{(n-B_d)/
 (2^{d-1}+(n-B_d)s_{d+1})}}.
$$
This is $O(1)$ provided $A$  satisfies
$A(n-B_d)>(2^{d-1}+(n-B_d)s_{d+1})n.$

Suppose next that $\bal\in I^{(2)}$. Then Lemma \ref{lem:iter}
produces a sequence of positive integers $q_j$, for $j\in \Delta$,  
which satisfy the conditions (\ref{ci}), (\ref{cii}) and (\ref{ciii}).
For each $j\in \Delta$ and $i\leq r_j$ we may choose $b_{i,j}\in \ZZ$
and $z_{i,j}\in \RR$, such that 
$$ 
\frac{q_j a_{i,j}}{q} =b_{i,j}+z_{i,j}
$$ 
with $|z_{i,j}|\leq Q_jP^{-j}$.
If there is a choice of $i,j$ for which 
${q_j a_{i,j}}\neq qb_{i,j}$, then we would be able to conclude that
$$
\frac{1}{qq_j}\leq \frac{|z_{i,j}|}{q_j}\leq \frac{Q_jP^{-j}}{q_j}\ll
\frac{L^{-s_j} P^{-j+\ve}}{q_j}, 
$$
by \eqref{Qform+}.
But then $L^{s_j}\ll qP^{-j+\ve}$, which once substituted into
\eqref{eq:52}, would show that  
$S(\a,q)\ll 1$ provided $A$  satisfies $jA-1>ns_j$.

We may therefore proceed under the assumption that 
${q_j a_{i,j}}=qb_{i,j}$ for every $j\in \Delta$ and every $i\leq
r_j$, or in other words, that $q_j\mathbf{a}^{(j)}=q\mathbf{b}^{(j)}$.
This implies that 
\[q\mid\gcd(qq_j\,,\,q\mathbf{b}^{(j)})=
\gcd(qq_j\,,\,q_j\mathbf{a}^{(j)})=q_j\gcd(q\,,\,\mathbf{a}^{(j)}).\]
Moreover, in view of 
(\ref{ci}),  we have $\gcd(q,\a^{(j)})\mid \gcd(q,\a^{(k)})$ when
$k>j$ and $k\in \Delta$. 
Thus 
\[q\mid q_j\gcd(q,\a^{(j)},\dots,\a^{(D)}),\]
for every $j\in \Delta$.  Applying \eqref{cii} and \eqref{Qform+} we
are therefore  
led to the conclusion that 
$$
\frac{q}{\gcd(q,\a^{(j)}, \dots,\a^{(D)})}\leq q_j
\leq Q_j =(\log P)^{e(j)} L^{-s_j},
$$
for every $j\in \Delta$. 
Noting that $(\log P)^{e(j)/s_j}\ll q^\ve$, 
this produces an upper bound for $L$ which we substitute into
\eqref{eq:52} to arrive at the statement of the lemma.\end{proof}

Using this result we may now  handle the singular series. Let 
$$
A(q)=
\sum_{\substack{\a \bmod{q}\\ \gcd(q,\a)=1}}
|S(\a,q)|.
$$
Let us put $d_j=\gcd(q,\a^{(j)},\dots,\a^{(D)})$ for each $j\in \Delta$. 
Suppose that $j_0$ is the least index $j\in \Delta$.
Then $d_{j_0}=1$ 
since $\gcd(q,\a)=1$. Moreover, we have 
$d_j\mid q$ for every $j\in \Delta$.
The number of $\a^{(j)}\bmod{q}$  associated to a given $d_j$  is 
$
(q/d_j)^{r_j}.
$
Moreover the total number of $d_1,\dots,d_D$ associated to a given $q$
is at most  
$\tau(q)^D=O(q^\ve)$.
Next we note that 
$$
\min_{j\in \Delta} \left(\frac{d_j}{q}\right)^{1/s_j}\leq 
\prod_{j\in \Delta}
 \left(\frac{d_j}{q}\right)^{\lambda_j/s_j},
 $$
 for any real numbers $\lambda_j\geq 0$ such that $\sum_{j\in \Delta}\lambda_j=1$.
We will apply this with 
\begin{equation}\label{eq:case}
\lambda_j = 
\begin{cases}
\theta+r_{j_0}s_{j_0}, & \mbox{if $j=j_0$,}\\
r_js_j, & \mbox{if $j\in \Delta\setminus \{j_0\}$,}
\end{cases}
\end{equation}
where 
$
 \theta=1-(s_1r_1+\dots+s_Dr_D).
$
In view of  our assumption
 \eqref{eq:ant}, such a choice is possible with  
$\theta\in (0,1)$.
It therefore follows from Lemma~\ref{lem:majorA} that 
\begin{align*}
A(q)
&\ll  q^{n+\ve/2} 
\sum_{d_1,\dots,d_D\mid q}
 \left(\frac{1}{q}\right)^{\theta/s_{j_0}}
\prod_{j\in \Delta} 
 \left(\frac{q}{d_j}\right)^{r_j}
\cdot \left(\frac{d_j}{q}\right)^{r_j}
\ll  q^{n-\theta/s_{j_0}+\ve}.
\end{align*}
Assuming  that $\theta/s_{j_0}>1$,
which is evidently  implied by \eqref{eq:ant}, 
this shows that the singular series is absolutely convergent and 
that \eqref{eq:show1} holds for an appropriate $\delta>0$.

We now turn to the exponential integral 
$J(\bga)$ in 
\eqref{eq:defn-J}, for general values of $\bga$.

\begin{lemma}\label{lem:majorB}
We have $J(\bga)\ll 1$ for any $\bga$. Moreover, for given 
 $\ve>0$, we have 
$$
J(\bga)\ll |\bga|^\ve
\min_{j\in \Delta} 
|\bga^{(j)}|^{-1/s_j},
$$
where $\bga^{(j)}=(\gamma_{1,j},\dots,\gamma_{r_j,j})$.
\end{lemma}

\begin{proof}
The estimate $J(\bga)\ll 1$ is trivial. We proceed to establish the bound 
$$
J(\bga)\ll  
|\bga^{(j)}|^{-1/s_j}|\bga|^\ve,
$$
for 
any $j\in \Delta$.
In doing so we may assume that $|\bga^{(j)}|>1$, since otherwise it
follows from the trivial bound. 

Our proof is analogous to the proof of Lemma \ref{lem:majorA}.  The
starting point is  \eqref{eq:51}, which we apply with  
$\bal=(P^{-d}\gamma_{i,d})$,  $\a=\mathbf{0}$ and $q=1$.  This gives
\begin{equation}\label{eq:53}
|J(\bga)|\leq P^{-n}|S(\bal)| + O(|\bga|P^{-1}) =
L + O(|\bga|P^{-1}).
\end{equation}
We take  $P=|\bga|^{A}$ for some large value of $A$ to 
be specified during the course of the proof.  
Our key ingredient is Lemma~\ref{lem:iter}.  
The case in which  $\bal\in I_d^{(1)}$, for some $d\in \Delta$, is
easily dispatched on taking  
$A$ to satisfy 
$A-1>1/s_j$ and $A(n-B_d)>(2^{d-1}+(n-B_d)s_{d+1})/s_j$.

It remains to consider the possibility that $\bal\in I^{(2)}$. Then
Lemma~\ref{lem:iter} produces a positive integer $q_j$ 
which satisfies the conditions  (\ref{cii}) and (\ref{ciii}).
For each $i\leq r_j$ we may choose $b_{i,j}\in \ZZ$
and $z_{i,j}\in \RR$, such that 
\begin{equation}\label{eq:54}
q_j \alpha_{i,j} =b_{i,j}+z_{i,j},
\end{equation}
with $\gcd(q_j,\mathbf{b}^{(j)})=1$ and  
$|z_{i,j}|\leq Q_jP^{-j}$.
If there is a choice of $i\leq r_j$ for which $b_{i,j}\neq 0$, 
then we would be able to conclude that
\begin{align*}
1 \leq |b_{i,j}|\leq  q_j|\alpha_{i,j}|+|z_{i,j}|&\leq  
q_jP^{-j}|\gamma_{i,j}|+Q_jP^{-j}, 
\end{align*}
whence
$$
1\leq Q_jP^{-j}|\gamma_{i,j}|+Q_jP^{-j} \ll Q_jP^{-j}|\bga^{(j)}| 
\ll L^{-s_j} P^{-j+\ve}|\bga^{(j)}|.
$$
This provides an upper bound for $L$, which once substituted into
\eqref{eq:53}, produces a satisfactory estimate for  
$J(\bga)$  provided that $A$ is chosen 
to satisfy $A-1>1/s_j$ and $A>2/j$. 

We proceed under the assumption that 
$b_{i,j}=0$ 
in \eqref{eq:54},
for every $i\leq r_j$.
But then $q_j=1$ and 
it follows that 
$$
P^{-j}|\gamma_{i,j}|=|\alpha_{i,j}| = |z_{i,j}| \leq Q_jP^{-j}=  
(\log P)^{e(j)}L^{-s_j}P^{-j}, 
$$
for every $i\leq r_j$.  Hence  $L\ll |\bga^{(j)}|^{-1/s_j}(\log P)^{e(j)/s_j}$.
Substituting this into \eqref{eq:53}, we easily  conclude the proof of
the lemma. 
\end{proof}

We now have everything in place to show that the singular integral 
converges. Recalling 
\eqref{21-si} and appealing to Lemma \ref{lem:majorB}, we find that
\begin{align*}
|\mathfrak{I}-\mathfrak{I}(H)|
&\leq \int_{|\bga|>H} |J(\bga)|\d\bga\\
&\ll \int_{|\bga|>H} |\bga|^{\ve/2} \min_{j\in \Delta} |\bga^{(j)}|^{-1/s_j} 
\d\bga.
\end{align*}
Let $N=\#\Delta$ and let $\t\in \RR_{>0}^N$. 
For given  $j\in \Delta$, the set
of $\bga^{(j)}\in \RR^{r_j}$ satisfying $|\bga^{(j)}|=t_j$
has $(r_j-1)$-dimensional measure  $O(t_j^{r_j-1})$. 
Hence 
\begin{align*}
|\mathfrak{I}-\mathfrak{I}(H)|
&\ll
\int_{\substack{\t\in \RR_{>0}^N\\
|\t|>H}}
|\t|^{\ve/2}   \min_{j\in \Delta} \{t_j^{-1/s_j} \}
\left(\prod_{j\in \Delta} t_j^{r_j-1}\right)
\d\t\\
&\leq
\int_{\substack{\t\in \RR_{>0}^N\\
|\t|>H}}
|\t|^{\ve}   \min_{j\in \Delta} \{t_j^{-1/s_j} \}
\left(\prod_{j\in \Delta} t_j^{r_j-1-\ve/(2N)}\right)
\d\t.
\end{align*}
We will consider the contribution to the right hand side 
from $\t$ for which $|\t|=t_{j_0}$, for some $j_0\in \Delta$.
If $H\geq 1$ we have 
$$
\min_{j\in \Delta} \{t_j^{-1/s_j} \}
\ll \min_{j\in \Delta} \{(1+t_j)^{-1/s_j} \} 
\le \prod_{j\in \Delta} (1+t_j)^{-\lambda_j/s_j}, 
 $$
 with $\lambda_j$ given  by 
 \eqref{eq:case}.
This therefore leads to the overall contribution
\begin{align*}
&\ll
\int_{\substack{\t\in \RR_{>0}^N\\
|\t|=t_{j_0}>H}}
\frac{t_{j_0}^{\ve-\theta/s_{j_0}}}
{\prod_{j\in \Delta} (1+t_j)^{1+\ve/(2N)}}  
\d\t
\ll 
\int_{H}^\infty 
t_{j_0}^{\ve-\theta/s_{j_0}-1}
\d t_{j_0}.
\end{align*}
This establishes \eqref{eq:show2} for a suitable $\delta>0$, as
required, provided only that $\theta>0$.  
Recalling that 
$
 \theta=1-(s_1r_1+\dots+s_Dr_D)
$,
this condition is ensured by \eqref{eq:ant}, which thereby completes
the proof of Lemma \ref{lem:major}.

\section{Proof of Theorem \ref{thm:crude}}\label{sec:crude}

We begin by disposing of the case in which $D$ is the only degree
present, so that $r_D=R$ and $\cD=RD$. In this situation
\[n_0=R(R+1)(D-1)2^{D-1}\]
as in Birch's result \cite{birch}.  
Thus Theorem \ref{thm:crude} is 
trivial in the case $D=1$, and for $D\ge 2$ we have to
show that $n_0+R-1\le R^2D^2 2^{D-1}$ and $n_0+R-1\le (RD-1)2^{RD}$.  However
\begin{align*}
R(R+1)(D-1)2^{D-1}+R-1&\le \{R(R+1)(D-1)+R\}2^{D-1}\\
&\le (2R^2(D-1)+R^2)2^{D-1}\\
&\le R^2D^2 2^{D-1}
\end{align*}
since $2D-1\le D^2$.  The first estimate then follows.  For the second
bound we observe that
\begin{align*}
R(R+1)(D-1)2^{D-1}+R-1&\le \{R(R+1)(D-1)+R-1\}2^{D-1}\\
&\le \{R(R+1)+R-1\}(RD-1)2^{D-1}
\end{align*}
since we are now supposing that $D\ge 2$.  However $R(R+1)+R-1\le
2^{2R-1}$ for any $R\ge 1$ and $2^{D-1+2R-1}\le 2^{RD}$ for $D\ge
2$. This establishes the second estimate.

We may assume henceforth that not all the forms have the same degree,
whence $R\ge 2$ and $D\ge 2$.  We also note that $D+R-1\le\cD\le DR-1$.
We now proceed to dispose of the case in which $n_0=n_0(D)$. We have
$n_0(D)=\cD 2^{D-1}$, so that we need to show that $\cD 2^{D-1}+R-1\le
\cD^2 2^{D-1}$ and $\cD 2^{D-1}+R-1\le
(\cD-1)2^{\cD}$.  We begin by observing that
\[\cD 2^{D-1}+R-1\le(\cD+R-1)2^{D-1}\le 2\cD 2^{D-1}.\]
The first bound then follows since $2\cD\le\cD^2$.  Moreover
$2\cD\le 4(\cD-1)$ and $D+1\le\cD$ whence
\[2\cD 2^{D-1}\le (\cD-1)2^{D+1}\le (\cD-1)2^{\cD},\]
as required for the second bound.

For the rest of our argument we examine $n_0(d)$ for $d<D$, and we
assume that $\#\Delta\ge 2$.  This allows us to set
$E=\max\{d\in\Delta:d<D\}$.
We begin by observing that
\[t_d=\sum_{k=d}^D 2^{k-1}(k-1)r_k\le 2^{D-1}\sum_{k=1}^D (k-1)r_k
= 2^{D-1}(\cD-R)\]
for every $d\ge 1$, whence
\begin{align*}
t_{d+1}+\sum_{j=d+1}^D t_jr_j&\le
2^{D-1}(\cD-R)\{1+\sum_{j=1}^Dr_j\}\\
&=2^{D-1}(\cD-R)(1+R).
\end{align*}
We also have
\[\cD_d\le \cD-D\le\cD-2 \quad \mbox{and} \quad  \cD_d\le E(R-1)\]
for $0\le d\le D-1$.  Thus
\[n_0(d) \le 2^{D-1}\{(\cD-2)(\cD-R+1)+(\cD-R)(1+R)\}\]
and 
\[n_0(d) \le 2^{D-1}\{E(R-1)(\cD-R+1)+(\cD-R)(1+R)\}\]
for $0\le d\le D-1$.

It will therefore suffice to show that
\[ 2^{D-1}\{(\cD-2)(\cD-R+1)+(\cD-R)(1+R)\}+R-1\le\cD^22^{D-1}\]
and
\[ 2^{D-1}\{E(R-1)(\cD-R+1)+(\cD-R)(1+R)\}+R-1\le(\cD-1)2^{\cD}.\]
For the first inequality we note that the left hand side is at most
\begin{align*}
2^{D-1}\{(&\cD-2)(\cD-R+1)+(\cD-R)(1+R)+R-1\}\\
&=2^{D-1}\{\cD^2-R^2+2R-3\}\hspace{3cm}\\
&\le2^{D-1}\cD^2.
\end{align*}
For the second inequality one sees that the left hand side is at most
\[ 2^{D-1}\{E(R-1)(\cD-R+1)+(\cD-R)(1+R)+R\},\]
and
\begin{align*}
E(R-1)(&\cD-R+1)+(\cD-R)(1+R)+R \\
&\le E(R-1)(\cD-1)+(\cD-1)(R+1)\\
&=\{ER-E+R+1\}(\cD-1)\\
&\le 2RE(\cD-1).
\end{align*}
To complete the argument we observe that $R\le 2^{R-1}$ and $E\le
2^{E-1}$, and that $2^{D-1+R-1+E-1}\le 2^{\cD-1}$ since $D+R+E-2\le \cD$.


\begin{thebibliography}{9}

\bibitem{baker}
R.C. Baker,
{\em Diophantine inequalities}. Oxford Science Publications, Oxford
University Press, New York, 1986.


\bibitem{ein}
A. Bertram, L. Ein and R. Lazarsfeld, Vanishing theorems, a theorem of
Severi, and the equations defining projective varieties. 
{\em J. Amer. Math. Soc.} {\bf 4} (1991), 587--602.


\bibitem{odd}
B.J. Birch,
Homogeneous forms of odd degree in a large number of variables.
{\em Mathematika} {\bf 4} (1957), 102--105.


\bibitem{birch}
B.J. Birch,
Forms in many variables.
{\em Proc. Roy. Soc. Ser. A} {\bf 265} (1961/62), 245--263.


\bibitem{BL}
T.D. Browning and D. Loughran, Varieties with too many rational points.
{\em Submitted}, 2013. (\texttt{arXiv:1311.5755})

\bibitem{bdhb}
T.D. Browning, R. Dietmann and D.R. Heath-Brown,
Rational points on intersections of cubic and quadric hypersurfaces.
{\em  J. Inst. Math. Jussieu}, to appear. (\texttt{arXiv:1309.0147})


\bibitem{bud}
J.-L.~Colliot-Th\'el\`ene, 
Points rationnels sur les fibrations.  
{\em Higher dimensional varieties and rational points (Budapest, 2001)},
171--221, Springer-Verlag,  2003.


\bibitem{Dav32}
H. Davenport, Cubic forms in thirty-two variables. {\em
Philos. Trans. Roy. Soc. London. Ser. A} {\bf 251} (1959), 193--232.

\bibitem{DavAM}
H. Davenport, {\em Analytic methods for Diophantine equations and 
Diophantine inequalities}. Second edition. 
Cambridge University Press, Cambridge, 2005.


\bibitem{D}
R. Dietmann, Weyl's inequality and systems of forms. 
{\em Submitted}, 2014.
({\tt arXiv:1208.1968v2})


\bibitem{manin}
{J. Franke, Y.I. Manin and Y. Tschinkel}, {Rational
points of bounded height on Fano varieties}. {\em Invent. Math.}
{\bf95} (1989), 421--435. 


\bibitem{harris} J. Harris, {\em Algebraic geometry}. Springer-Verlag,
New York, 1992.

\bibitem{HMP}
J. Harris, B. Mazur and R. Pandharipande,
Hypersurfaces of low degree.
{\em Duke Math. J.} {\bf  95} (1998), 125--160.

\bibitem{har70}
R. Hartshorne, 
{\em Ample subvarieties of algebraic varieties}. Lecture Notes in Math. {\bf 156}, Springer-Verlag, Berlin, 1970.

\bibitem{hart-conj}
R. Hartshorne, Varieties of small codimension in projective
space. {\em Bull. Amer. Math. Soc.} {\bf 80} (1974), 1017--1032. 

\bibitem{hart}
R. Hartshorne, {\em Algebraic Geometry}.   Springer-Verlag, New York,
1977.

\bibitem{hooley}
C. Hooley, On some topics connected with Waring's problem.
{\em J. Reine Angew. Math.} {\bf 369} (1986), 110--153. 

\bibitem{PPW} 
S.T. Parsell, S.M. Prendiville and T.D. Wooley,  Near-optimal mean value estimates for multidimensional Weyl sums. {\em Geom. Funct. Anal.} {\bf 23} (2013), 1962--2024.


\bibitem{peyre}
{E. Peyre}, {Hauteurs et mesures de Tamagawa sur les
vari\'{e}ti\'{e}s de Fano}. {\em Duke Math. J.} {\bf79}
 (1995), 101--218.


\bibitem{PV}
B. Poonen and J.F. Voloch, Random Diophantine equations. 
{\em Arithmetic of higher-dimensional algebraic varieties (Palo Alto,
CA, 2002)}, 175--184,
Progr. Math. {\bf 226}, Birkh\"auser, 2004.


\bibitem{S}
D. Schindler, A variant of Weyl's inequality for systems of forms and
applications. {\em Submitted}, 2014. ({\tt arXiv:1403.7156})


\bibitem{schmidt} W. Schmidt, The density of integer points on
  homogeneous varieties. {\em Acta Math.} {\bf 154} (1985),  243--296.



\end{thebibliography}
\end{document}